\documentclass[]{amsart}
\usepackage{amssymb,mathdots}
\usepackage{mathrsfs}
\usepackage[utf8]{inputenc}
\usepackage{amsmath, graphicx}
\usepackage{wrapfig}
\usepackage{enumerate}
\usepackage{url}

\def\scaledpicture#1by#2(#3scaled#4){{
\dimen0=#1  \dimen1=#2
\divide\dimen0 by 1000 \multiply\dimen0 by #4
\divide\dimen1 by 1000 \multiply\dimen1 by #4
\picture \dimen0 by \dimen1 (#3 scaled #4)}}
\def\dfigure#1by#2(#3scaled#4offset#5:#6)
    {\medskip
     \vglue 2mm minus 2mm
     $$
       \hbox{
         \hglue#5
         {\scaledpicture #1 by #2 (#3 scaled #4)}
       }
     $$
     \par\goodbreak
     \vglue 2mm minus 2mm
     \medskip}

\def\qmod#1#2{{\hbox{}^{\displaystyle{#1}}}\!\big/\!\hbox{}_{
\displaystyle{#2}}}

\def\resto#1#2{{
#1\hskip 0.4ex\vline_{\hskip 0.2ex\raisebox{-0,2ex}
{{${\scriptstyle #2}$}}}}}

\def\C{{\mathbb C}}

\def\F{{\mathbb F}}

\def\N{{\mathbb N}}

\def\P{{\mathbb P}}

\def\R{{\mathbb R}}
\def\Z{{\mathbb Z}}

\def\union{\mathop{\bigcup}}

\def\textmap#1{\mathop{\vbox{\ialign{
                                  ##\crcr
      ${\scriptstyle\hfil\;\;#1\;\;\hfil}$\crcr
      \noalign{\kern 1pt\nointerlineskip}
      \rightarrowfill\crcr}}\;}}
\def\bigtextmap#1{\mathop{\vbox{\ialign{
                                  ##\crcr
      ${\hfil\;\;#1\;\;\hfil}$\crcr
      \noalign{\kern 1pt\nointerlineskip}
      \rightarrowfill\crcr}}\;}}
      
\newcommand{\cal}{\mathcal}
\def\textlmap#1{\mathop{\vbox{\ialign{
                                  ##\crcr
      ${\scriptstyle\hfil\;\;#1\;\;\hfil}$\crcr
      \noalign{\kern-1pt\nointerlineskip}
      \leftarrowfill\crcr}}\;}}

\def\fg{{\mathfrak f}}
\def\g{{\mathfrak g}}

\def\pg{{\mathfrak p}}

\def\sg{{\mathfrak s}}

\def\Sg{{\mathfrak S}}

\newtheorem{sz}{Satz}[section]
\newtheorem{thry}[sz]{Theorem}
\newtheorem{pr}[sz]{Proposition}
\newtheorem{re}[sz]{Remark}
\newtheorem{co}[sz]{Corollary}
\newtheorem{dt}[sz]{Definition}
\newtheorem{lm}[sz]{Lemma}

\def\U{\mathrm{U}}

\def\deg{\mathrm {deg}}
\def\Hom{\mathrm{Hom}}

\def\id{ \mathrm{id}}

\def\im{\mathrm{im}}
\def\rk{\mathrm {rk}}

\def\U2{\mathrm{U(2)}}
\def\niq{=\kern-.18cm /\kern.08cm}
\def\oo{{\scriptstyle{\cal O}}}

\def\mult{\mathrm{mult}}

\begin{document}

\title{A wall crossing formula for degrees of real central projections}
\author{Christian Okonek \and  Andrei Teleman}
\begin{abstract}
The main result is a wall crossing formula for central projections defined on submanifolds of a real projective space. Our formula gives the jump of the degree of such a projection when the center of the projection varies. The fact that the degree depends on the projection is a new phenomenon, specific to real algebraic geometry. We illustrate this phenomenon in many interesting situations. The crucial assumption on the class of maps we consider is relative orientability, a condition which allows us to define a $\Z$-valued degree map in a coherent way. We end the article with several examples, e.g. the pole placement map associated with a quotient, the Wronski map, and a new version of the real subspace problem.
\end{abstract}

\maketitle

\setcounter{section}{-1}

\section{Introduction}

One of the fundamental problems in real algebraic geometry concerns the  solutions of systems of real algebraic equations. Since, in general, even the existence of real solutions is not guaranteed, it is important to find a  priori lower bounds for the number of these solutions.  

In the last years  several important developments have taken place in this direction, which are related to problems in enumerative geometry \cite{DeK},   \cite{IKS}, \cite{FK}, \cite{W1}, \cite{W2}, \cite{OT2}, or to the study of  certain polynomial systems which often have interesting applications \cite{EG1}, \cite{EG2}, \cite{SS}.  In many cases these lower bounds are provided by the degrees of certain maps, for instance  central projections of projective submanifolds \cite{EG1}, \cite{EG2}, \cite{So}, \cite{SS}. For a projective $m$-dimensional submanifold $X\subset \P^{N-1}_\C$, and a central projection $\pi:\P^{N-1}_\C\dasharrow\P^m_\C$  whose center does not intersect $X$, one gets a  finite map $\resto{\pi}{X}:X\to\P^m_\C$ whose  degree  with respect to complex orientations  coincides with the number of points of a  fibre $(\resto{\pi}{X})^{-1}(p)$   if multiplicities are taken into account.   Note  that this fibre  can be   regarded as the set of solutions of a system of  homogeneous algebraic equations.  This degree can also be identified with the number of points of the intersection $X\cap H$ for a general codimension $m$ projective subspace $H\subset \P^{N-1}_\C$; hence it is cohomologically determined and independent of the choice of the central projection $\pi$ (as long as its center does not intersect $X$).  This  is an important example of the fundamental principle of  {\it conservation of  numbers}. It is well known that this principle does not hold in real algebraic geometry. One of the goals of this article is to show that  in real geometry, this principle should be replaced by a {\it wall crossing formula}. Such wall crossing formulae play an important role in gauge theory \cite{DoK} \cite{OT1}, but apparently the wall crossing phenomenon (jump of an invariant in a well controlled way) has not been studied until now in real algebraic geometry. The general principle of wall crossing is very simple: one has a parameter space ${\cal P}$ which parameterizes a certain class of maps, and a wall ${\cal W}\subset {\cal P}$ of {\it bad} maps. If certain conditions are fulfilled,  one can define a degree map
$$\deg:\pi_0({\cal P}\setminus {\cal W})\to\Z
$$
which associates to any chamber  $C\in \pi_0({\cal P}\setminus {\cal W})$ an integer. A wall crossing formula computes the difference $\deg(C_+)-\deg(C_-)$ between the integers associated to two adjacent  chambers.  In this article we will prove such a  wall crossing formula  for the space of central projections restricted to  a smooth, {\it  not necessarily algebraic} submanifold 
$X$ of a real projective space. We are not aware of any systematic results in this generality. Note  in particular  that standard algebro-geometric techniques like e.g. elimination theory, do not apply.

The  crucial assumption on the class of maps we consider is  {\it relative orientability}, a condition which allows us to define a $\Z$-valued degree map in a coherent way.  This wall crossing phenomenon for real central projections $\P(V)\supset X\to\P(W)$ pointed out by our result is in striking contrast to the invariance of the degree  of a section  in a relatively oriented vector bundle  (see \cite{OT2}). We describe now briefly  the results  of this article: \\

In the first section we introduce the important notion of relative orientation of a map between topological manifolds,  and we define the degree of a relatively oriented map between closed manifolds of the same dimension.  For relatively {\it orientable} maps $f$ (so for maps which  can be relatively oriented, but have not been endowed with a relative orientation) one can define the absolute degree $|\deg|(f)$, which is similar to Kronecker's concept of  {\it characteristic} and Hopf's {\it absolute degree}.   In the differentiable case the degree of a relatively oriented map can be computed using the local degrees at the points of a finite fibre. The special case of a Real finite  holomorphic map $f:X\to Y$  between Real complex manifolds of the same dimension is particularly important. If the restriction
$$f(\R):X(\R)\to Y(\R)
$$
 is relatively orientable, then there are fundamental estimates and comparison formulae for the sum of the multiplicities of the points of a real fibre $f(\R)^{-1}(y)$, $y\in Y(\R)$:
 $$|\deg|(f(\R))\leq \sum_{x\in f(\R)^{-1}(y)} \mult_x f(\R)\leq \deg(f)\ ,
 $$
 $$|\deg|(f(\R))\equiv \sum_{x\in f(\R)^{-1}(y)} \mult_x f(\R)\equiv \deg(f)\  ({\rm mod}\ 2)\ .
 $$

 In the second section  we study central projections. Let $V$, $W$ be real vector spaces, $f\in\Hom(V,W)$ and
 $$[f]:\P(V)\setminus\P(\ker(f))\to\P(W)
 $$
the induced morphism. Let $X\subset \P(V)$ be a compact submanifold with $\dim(X)=\dim(\P(W))$ and $X\cap\P(\ker(f))=\emptyset$. The induced map
$$[f]_X:X\to\P(W)
$$
is relatively orientable if and only if  
$$w_1(X)=(\dim(X)+1) w_1(\lambda_{V,X})\ ,$$
where $\lambda_{V,X}:=\resto{\lambda_V}{X}$ denotes the restriction to $X$ of  the tautological line bundle  $ \lambda_V$  of $\P(V)$. It follows in particular  that the relative orientability of such central projections $[f]_X$ is independent of $f$. 

The tensor product  $T_X\otimes \lambda_{V,X}$   can be written as $Y_X/\lambda_{V,X}$, where $Y_X$ is  a  subbundle  of the trivial bundle $\underline{V}_X:=X\times V$.
With this definition we obtain a canonical identification 
$$
T_X=\Hom\big(\lambda_{V,X}, \qmod{ Y_X}{\lambda_{V,X}}\big)\ .
$$

 The data of a  relative orientation of a map $f\in \Hom(V,W)$  with $X\cap\P(\ker(f))=\emptyset$  is equivalent to the data of a bundle isomorphism
$$\mu: X\times\det(W)\to \det (Y_X)\ .
$$
In our situation the space ${\cal P}$ parameterizing the relevant maps is the vector space $\Hom(V,W)$,  and the wall associated with the submanifold $X\subset \P(V)$ is:
$${\cal W}_X:=\{f\in\Hom(V,W)|\ X\cap \P(\ker(f))\ne\emptyset\}\ .
$$

A point  $f_0\in{\cal W}_X$  on the wall is called regular when $f_0$ is surjective, the intersection $X\cap \P(\ker(f_0))$ consists  only of one point $\xi_0$, and $\ker(f_0)\cap Y_{\xi_0}=\xi_0$. 
%%%%%%%%%%%%%%%%%%%%%%%%%%%%
%%%%%%%%%%%%%%%%%%%%%%%%%%%%
%
Let ${\cal W}_X^0$ be the subspace of regular points in ${\cal W}_X$. Our first result in section \ref{Wall} shows that the wall ${\cal W}_X\subset\Hom(V,W)$ is a smooth hypersurface in all regular points $f_0\in {\cal W}_X^0$ and identifies the normal line $N_{{\cal W}^0_X,f_0}$ with $\Hom(\det(Y_{\xi_0}),\det(W))$ via a canonical isomorphism 
$$\psi_{f_0}:N_{{\cal W}^0_X,f_0}\to\Hom(\det(Y_{\xi_0}),\det(W))\ .$$
This is not a standard result, since the manifold $X$ is not supposed to be algebraic. It also  shows that the choice of an orientation parameter $\mu:X\times \det(W)\to \det(Y_X)$  serves a second, completely different purpose: $\mu_{\xi_0}^{-1}:\det(Y_{\xi_0})\to\det(W)$ defines a generator of the normal line  $N_{{\cal W}^0_X,f_0}$ for every $f_0\in {\cal W}_X^0$ with $\ker(f_0)\cap Y_{\xi_0}=\xi_0$.

The main result in this section is the wall crossing formula (see Theorem \ref{wall-cross-th}):
\\ \\
{\bf Theorem}   (Wall-crossing formula) {\it Let  $f_0\in \mathcal{W}_X^0$  be a regular  point on the wall with $\ker(f_0)\cap Y_{\xi_0}=\xi_0$. Consider a smooth map   
$$\fg :N_{{\cal W}^0_X,f_0}\to\Hom(V,W)$$
whose differential $\fg_{*,0}$ is a right splitting of the exact sequence
$$0\to  T_{\mathcal{W}_X^0,f_0}\to  T_{\Hom(V,W),f_0}\to N_{{\cal W}^0_X,f_0}\to 0\ .$$
Then for every sufficiently small $\tau\in N_{{\cal W}^0_X,f_0}\setminus\{0\}$ we have: 
\begin{enumerate}
\item $\fg(\tau)\in\Hom(V,W)\setminus {\cal W}_X$ and  $[\fg(\tau)]_X$ is a local diffeomorphism at $\xi_0$.
\item  The local degree of $[\fg(\tau)]_X$ at $\xi_0$ is 
$$\deg_{\nu(\fg(\tau),\mu),\xi_0}([\fg(\tau)]_X)=\mathrm{sign}(\mu_{\xi_0}\circ(\psi_{f_0}(\tau)))\ .$$
\item $\deg_{\nu(\fg(\tau),\mu)}([\fg(\tau)]_X)-\deg_{\nu(\fg(-\tau),\mu)}([\fg(-\tau)]_X)=2\mathrm{sign}(\mu_{\xi_0}\circ(\psi_{f_0}(\tau)))$.
\end{enumerate}
}
\vspace{3mm}

As a consequence we obtain a general formula which computes the degree difference $\deg_{\nu(g_1,\mu)}([g_1]_X)-\deg_{\nu(g_0,\mu)}([g_0]_X)$ for a smooth path $(g_t)_{t\in[0,1]}$ in $\Hom(V,W)$ with $g_0$, $g_1\in \Hom(V,W)\setminus {\cal W}_X$, which intersects the wall ${\cal W}_X$ only in regular points with transversal intersection.  

The third result in section \ref{Wall} describes the irregular locus ${\cal W}_X\setminus {\cal W}_X^0$ of the wall. We show that ${\cal W}_X\setminus {\cal W}_X^0$ is closed and its complement in $\Hom(V,W)$ is connected. This implies that any two points $g_0$, $g_1\in \Hom(V,W)\setminus {\cal W}_X$ can be connected by a smooth path $(g_t)_{t\in[0,1]}$ which intersects  ${\cal W}_X$ only along  
${\cal W}_X^0$  with transversal intersection. In other words, the difference $\deg_{\nu(g_1,\mu)}([g_1]_X)-\deg_{\nu(g_0,\mu)}([g_0]_X)$ can always be  computed by our difference formula.  

Important applications of our difference formula are the following general properties of the degree map
$$\deg_\mu^X:\pi_0(\Hom(V,W)\setminus{\cal W}_X)\to \Z\ .$$
\begin{enumerate}
\item All  values of the degree map   are congruent modulo 2.
\item If $a\in \im(\deg_\mu^X)$,  then any integer $c$ with $-|a|\leq c\leq |a|$ which has the same parity as $a$ also belongs to $\im(\deg_\mu^X)$.
\end{enumerate}

The second ``no gaps" property can be considered as a strong existence result; it shows that the phenomenon exhibited by the theorem of Brockett and Segal (\cite{Se} p. 41, see also section \ref{DegRatFunct} below) is completely general.

%%%%%%%%%%%%%%%%%%%%%%%%
%%%%%%%%%%%%%%%%%%%%%%%%
 
In the third section we discuss  examples.  First we describe a large class of important Real complex manifolds, the so called conjugation manifolds \cite{HHP}, for which the orientability of $f(\R)$ can be easily checked. This class contains   Grassmann manifolds and  toric manifolds with their standard Real structures.  Then we discuss Wronski projections, the universal pole placement map, and a real subspace problem. Note that in many interesting situations one has a canonical relative orientation of the considered projections.  In these situations one obtains canonical signs corresponding to the points  of a regular fibre.

\section{Degrees of real maps}

By topological manifold we always mean a topological manifold which is Hausdorff and paracompact.

 Let $M$ be an $n$-dimensional  topological manifold, and $\oo_M$ its orientation sheaf; for every point $m\in M$ we have 
$$\oo_{M,m}:=H_n(M,M\setminus\{m\},\Z)\ .
$$
 Note that for every locally constant sheaf $\xi$ on $M$ and for every open neighborhood $U$ of $m$ one has canonical identifications  
$$H_n(M,M\setminus\{m\},\xi)=H_n(U,U\setminus\{m\},\resto{\xi}{U})=H_n(U,U\setminus\{m\},\xi_m)=$$
$$=H_n(M,M\setminus\{m\},\xi_m)=\oo_{M,m}\otimes\xi_m\ ,
$$
and 
$$H^n(M,M\setminus\{m\},\xi)=H^n(U,U\setminus\{m\},\resto{\xi}{U})=H^n(U,U\setminus\{m\},\xi_m)=$$
$$=H^n(M,M\setminus\{m\},\xi_m)=\Hom(\oo_{M,m},\xi_m)\ .
$$

Suppose now that $M$ is connected and closed. The canonical fundamental class of $M$ in cohomology is the canonical generator  $\{M\}$ of  $H^n(M,\oo_M)$; for every $m\in M$  this class can be written as
$$\{M\}=j_{m}(c_m)
$$
where  $j_m:H^*(M,M\setminus\{m\},\oo_M)\to H^*(M,\oo_M)$ is the natural morphism, and $c_m=\id_{\oo_{M,m}}$ is the canonical generator of $H^n(M,M\setminus\{m\},\oo_M)$.

Similarly, the canonical class of $M$ in homology is the class  $[M]\in H_n(M,\oo_M)$ whose image in $H_n(M,M\setminus\{m\},\oo_M)$ is the canonical generator of this group for every $m\in M$.

\begin{dt} Let $M$, $N$ be differentiable manifolds, and   $g:M\to N$ a continuous map. A  relative orientation of $g$ is an isomorphism $\nu:g^*(\oo_N)\to \oo_M$. A   relatively oriented map is a pair $(g,\nu)$, where $\nu$ is a relative orientation of $g$. A continuous map $g:M\to N$ is called relatively orientable if it admits a relative orientation.
\end{dt}
Since $w_1$ classifies real line bundles on paracompact spaces, we obtain:
\begin{re}\label{w} A continuous map $g:M\to N$  between differentiable manifolds is relatively orientable if and only if $g^*(w_1(T_N))=w_1(T_M)$.
\end{re}
\begin{re} \label{H1} If $H^1(M,\Z_2)=0$, then any continous map $g:M\to N$ is   relatively orientable. In particular any  map $g:S^n\to N$ ($n\geq 2$) is  relatively orientable. 
\end{re}

Note that if $\dim(M)=\dim(N)=:n$ and $(g,\nu)$ is a relatively oriented differentiable map $M\to N$ between closed  connected manifolds, one can define the degree $\deg_\nu(g)$ by the formula
\begin{equation}\label{degDef}\deg_\nu(g)\{M\}=\nu_*(g^*(\{N\}))\ .
\end{equation}

If one just knows that $g$ is relatively orientable, but did not fix a relative orientation of it, one can still define  the {\it  absolute degree}  $|\deg|(g)$ of $g$ by
$$|\deg|(g):=|\deg_\nu(g)|
$$
for any relative orientation $\nu$ of $g$.

Let $x\in M$ be an isolated point in the fibre over its image $y:=g(x)$, and choose charts $h:(U,x)\to (U',0)\subset(\R^n,0)$, $k:(V,y)\to (V',0)\subset(\R^n,0)$ such that the   orientations of $T_xM$, $T_yN$ defined by the two charts correspond via $\nu_x$. Then we define
$$\deg_{\nu,x}(g):=\deg_0(k\circ g\circ h^{-1})\ ,
$$
where the right hand term is computed with respect to the standard orientation of $\R^n$. This degree can be defined in an intrinsic way using the map 
$$g_x:(U,U\setminus\{x\})\to (N,N\setminus\{y\})$$
 (which is well defined when $U$ is sufficiently small) and writing  
\begin{equation}\label{local}
 \nu_x(e_xg_x^*(c_y))=\deg_{\nu,x}(g) c_x\ ,
\end{equation}
where $e_x:H^n(U,U\setminus\{x\},\oo_y)\to H^n(M,M\setminus\{x\},\oo_y)$ is  the isomorphism defined by excision. 
\begin{pr} \label{degformula} Let $M$, $N$ be closed connected $n$-manifolds,  $g:M\to N$ a smooth map, $\nu:g^*(\oo_N)\to\oo_M$  a relative orientation, and $y\in N$  a point with $g^{-1}(y)$ finite. Then
\begin{equation}\label{sum}\deg_\nu(g)=\sum_{x\in g^{-1}(y)} \deg_{\nu,x} (g)\ .
\end{equation}
\end{pr}
\begin{proof}
  
Write $\{N\}=j_y(c_y)$, where $c_y\in H^n(N,N\setminus\{y\},\oo_N)$ is the canonical generator. One gets a pull-back class 
$$g^*(c_y)\in H^n(M,M\setminus g^{-1}(y) ,g^*(\oo_N))\ ,
$$
and $g^*(\{N\})$ can be written as
$$g^*(\{N\})=J_{M,y}(g^*(c_y))\ ,
$$
where $J_{M,y}:H^n(M,M\setminus g^{-1}(y) ,g^*(\oo_N))\to H^n(M,g^*(\oo_N))$ is the canonical map.

For every $x\in g^{-1}(y)$ let $U_x$ be an open  neighborhood of $x$ such that  $U_{x_1}\cap U_{x_2}=\emptyset$ for $x_1\ne x_2$. Put $U:=\union_{x\in g^{-1}(y)} U_x$, denote by $g_x:(U_x, U_x\setminus\{x\})\to (N,N\setminus\{y\})$ the  restriction of $g$, and note that the inclusion  $U\hookrightarrow M$ induces an   isomorphism 
$$H^n(U,U\setminus g^{-1}(y) ,(\resto{g}{U})^*(\oo_N))\to H^n(M,M\setminus g^{-1}(y) ,g^*(\oo_N))\ ,\ $$
by excision.  We denote by 
$$J_{U,y}:H^n(U,U\setminus g^{-1}(y) ,(\resto{g}{U})^*(\oo_N))\to H^n(M,g^*(\oo_N)) $$
  the composition of $J_{M,y}$ with this isomorphism. One has an obvious isomorphism
$$H^n(U,U\setminus g^{-1}(y) ,(\resto{g}{U})^*(\oo_N))=\bigoplus_{x\in g^{-1}(y)} H^n(U_x,U_x\setminus \{x\} ,\oo_y)
$$
and the restriction of  $J_{U,y}$ to a summand $H^n(U_x,U_x\setminus \{x\} ,\oo_y)$ is the composition $j_x\circ e_x$ of the  isomorphism $e_x:H^n(U_x,U_x\setminus\{x\},\oo_y)\to H^n(M,M\setminus\{x\},\oo_y)$  with the canonical map $J_x:H^*(M,M\setminus\{x\},\oo_y)\to H^*(M,g^*(\oo_N))$. Therefore
$$\nu g^*(\{N\})=\nu J_{U,y}((\resto{g}{U})^*(c_y))=\nu J_{U,y}\big( \sum_{x\in g^{-1}(y)}g_x^*(c_y)\big)= \nu\big(\hspace{-1mm}\sum_{x\in g^{-1}(y)} J_x\circ e_x\big(  g_x^*(c_y)\big) \big)
$$
$$=\sum_{x\in g^{-1}(y)} \nu_x\circ J_x\circ e_x\big( g_x^*(c_y)\big)=\sum_{x\in g^{-1}(y)} j_x\circ \nu_x \circ e_x\big( g_x^*(c_y)\big)=\sum_{x\in g^{-1}(y)} \deg_{\nu,x} (g) j_{x} (c_x)  $$
$$=
\big\{\sum_{x\in g^{-1}(y)} \deg_{\nu,x} (g)  \big\}\{M\}\ . 
$$
by (\ref{local}).
\end{proof}
\begin{pr}\label{estimates} Let $X$, $Y$ be compact complex manifolds endowed with Real structures,  let $f:X\to Y$ be a  finite Real   holomorphic map such that the induced map $f(\R):X(\R)\to Y(\R)$ is relatively oriented. Then
\begin{enumerate}
\item \label{firstEst} For any $y\in Y(\R)$ one has
$$|\deg|(f(\R))\leq \sum_{x\in f(\R)^{-1}(y)} \mult_x f(\R)\leq \deg(f)\ .$$
\item $\deg(f(\R))\equiv \deg(f)$ mod 2.
\item For any $y\in Y(\R)$ one has
$$\deg(f(\R)))\equiv\sum_{x\in f(\R)^{-1}(y)}\mult_x f(\R)\hbox{ (\rm mod 2) }\ .$$
\end{enumerate}
\end{pr}

\begin{proof} Choose a relative orientation $\nu:f(\R)^*(\oo_{Y(\R)})\to\oo_{X(\R)}$.\\  

(1) For the first inequality note that
$$|\deg|(f(\R))=|\sum_{x\in f(\R)^{-1}(y)}\deg_{\nu,x} (f(\R))|\leq \sum_{x\in f(\R)^{-1}(y)}\mult_x f(\R)=$$%
$$=\sum_{x\in f(\R)^{-1}(y)}\mult_x f \ .
$$

For the second inequality, note that  for any point $y\in Y(\R)$, we have 
\begin{equation}\label{degf} \deg(f)=\sum_{x\in f(\R)^{-1}(y)}\mult_x f+\sum_{x\in f^{-1}(y)\setminus  X(\R)}\mult_x f=$$
$$=\sum_{x\in f(\R)^{-1}(y)}\mult_x f(\R)+\sum_{x\in f^{-1}(y)\setminus  X(\R)}\mult_x f
\end{equation}
and all the terms on the right are positive.
\\ \\
(2) Choose a regular value $y\in Y(\R)$ of $f(\R)$. Then one has $\deg_{\nu,x} (f(\R))\in\{\pm 1\}$ for every $x\in f^{-1}(y)$ and
$$\deg(f(\R))=\sum_{x\in f(\R)^{-1}(y)} \deg_{\nu,x} (f(\R))\equiv \sum_{x\in f(\R)^{-1}(y)} \mult_x f(\R)\hbox{ (mod } 2)\ .
$$
It suffices to note that 
$$\mult_x f(\R)=\mult_x f\ ,\ \hbox{and}\ \sum_{x\in f^{-1}(y)\setminus  X(\R)}\mult_x f\equiv 0\hbox{ (mod 2 })\ .$$
(3) We use (2) and note that, by (\ref{degf}) one obviously has
$$\deg(f)\equiv \sum_{x\in f(\R)^{-1}(y)}\mult_x f(\R)\hbox{ (mod } 2)\ .
$$
\end{proof}

Note that the estimate (1) holds without any transversality assumption on $f$. 
A  statement similar to Proposition \ref{estimates} holds for Real sections in Real holomorphic vector bundles over compact complex manifolds endowed with Real structures \cite{OT2}.
\begin{pr} Let $X$ be compact complex manifold endowed with Real structure, and let $E\to X$ be a Real holomorphic vector bundle over $X$ with $\rk(E)=\dim(X)=n$. Suppose that the real vector bundle $E(\R)\to X(\R)$ is relatively orientable, and let $s$ be a Real holomorphic section of $E$ with finite zero locus $Z(s)$. Then
\begin{enumerate}
\item $$|\deg|(E(\R))\leq \sum_{z\in Z(s)\cap X(\R)}\mult_z(s)\leq\langle c_n(E),[X]\rangle\ .$$
\item $$|\deg|(E(\R))\equiv \langle c_n(E),[X]\rangle\   \hbox{ (mod 2)}\ .$$
\item $$|\deg|(E(\R))\equiv \sum_{z\in Z(s)\cap X(\R)}\mult_z(s)\ .
$$
\end{enumerate}
\end{pr}

\section{Wall crossing for degrees of real projections}\label{Wall}

\subsection{Projections. The wall associated with a submanifold of $\P(V)$}

Let $V$ be an $N$-dimensional real vector space and let $\lambda_V$ be the tautological line bundle on the projective space $\P(V)$. By definition, $\lambda_V$ is a line subbundle of the trivial bundle $\underline{V}:=\P(V)\times V$ and the tangent bundle $T_{\P(V)}$ can be canonically identified with $\Hom(\lambda_V,\underline{V}/\lambda_V)=\lambda_V^\vee\otimes\underline{V}/\lambda_V$.

Let  $X\subset\P(V)$ be a compact submanifold of dimension $m<N-1$. We  denote by $T_X$ the tangent bundle of $X$ regarded as a subbundle of the restriction $\resto{T_{\P(V)}}{X}$, and by $\underline{V}_X$, $\lambda_{V,X}$ the restrictions of the bundles $\underline{V}$, $\lambda_V$ to $X$.  The tensor product $T_X\otimes \lambda_{V,X}$ is a subbundle of  the quotient bundle $ \underline{V}_X/\lambda_{V,X}$, so it can be written as $Y_X/\lambda_{V,X}$, where $Y_X\subset \underline{V}_X$ is  the subbundle of  $\underline{V}_X$ defined as  the preimage of $T_X\otimes \lambda_{V,X}$  under the epimorphism 
$$\underline{V}_X\ \twoheadrightarrow \  \qmod{\underline{V}_X}{\lambda_{V,X}}\  .$$

With this definition we obtain a canonical identification 
\begin{equation}\label{tangentX}
T_X=\Hom\big(\lambda_{V,X}, \qmod{ Y_X}{\lambda_{V,X}}\big)\ ,
\end{equation}
which will play an important role in the following constructions. Note that the bundle $Y_X$ can be identified with the dual jet bundle $\left[J^1( {\lambda_{V,X}^\vee})\right]^\vee$ of   $ {\lambda_{V,X}^\vee}$.\\

Let  now $W$ be a real vector space of dimension $m+1$.  A morphism $f\in\Hom(V,W)$ defines the central projection
$$[f]:\P(V)\setminus\P(\ker(f))\to\P(W)\ ,
$$
 whose restriction to $X\setminus \P(\ker(f))$ will be denoted by  $[f]_X$.  

The differential of a central  projection $[f]$ can be computed as follows: For a line $l\in\P(V)\setminus\P(\ker(f))$  we have an induced isomorphism 
$$r_{f,l}:l\to f(l)\ , \ 
$$ 
and an induced  linear map
$$q_{f,l}:V/l\to W/f(l)\ .
$$

Let $H'$ be a linear complement of the line $f(l)$ in $W$, and note that $H:=f^{-1}(H')$  is a linear complement of  $l$ in $V$.  For $\varphi\in\Hom(l,H)$, the $[f]$-image of the line $l_\varphi:=\{v+\varphi(v)|\ v\in l\}$ in $\P(W)$ is $\{f(v)+f(\varphi(v))|\ v\in l\}$, which is the graph of $f\circ\varphi\circ r_{f,l}^{-1}$.  This shows that via the identifications
$$T_l(\P(V))=\Hom(l,V/l)\ ,\ T_{f(l)}(\P(W))=\Hom(f(l),W/f(l))
$$
the differential $[f]_{*,l}$ at a point $l\in \P(V)\setminus\P(\ker(f))$ is given by 
\begin{equation}\label{old-remark}[f]_{*,l}(\varphi)=q_{f,l}\circ\varphi\circ r_{f,l}^{-1}\ .
\end{equation}

The differential  of the restriction $[f]_X$   is given by the same formula applied to $\varphi\in T_l(X)=\Hom(l,Y_l/l)$. More precisely: 
\begin{re}\label{diffX} The differential  of the restriction $[f]_X$ at $l\in X\setminus\P(\ker(f)) $  is given by the formula:
\begin{equation}\label{old-remark-new}([f]_{X})_{*,l}(\varphi)=p_{f,l}\circ\varphi\circ r_{f,l}^{-1}\ .
\end{equation}
where $p_{f,l}:Y_l/l\to W/ f(l)$ is   the morphism induced by $f$.
\end{re}

\begin{re} \label{localiso}   Let $l\in X\setminus\P(\ker(f))$. The following conditions are equivalent:
\begin{enumerate}
\item $[f]_X$ is a local diffeomorphism at $l$,
\item $f^{-1}(f(l))\cap Y_{l}=l$,
\item $\ker(f)\cap Y_{l}=\{0\}$.
\end{enumerate}
\end{re}
\begin{proof}

Using Remark  \ref{diffX} we see that $([f]_X)_{,*l}$ is an isomorphism if and only if the restriction of   
$$q_{f,l}:V/l\to W/f(l)$$
    to $Y_{l}/l$ is an isomorphism. Since the  kernel of this restriction is $\big[f^{-1}(f(l))\cap Y_{l}\big]/l$, we obtain the equivalence (1)$\Leftrightarrow $(2). On the
 other hand $f^{-1}(f(l))=l+\ker(f)$, hence $f^{-1}(f(l))\cap Y_l=(l+\ker(f))\cap Y_l$. The composition
$$\ker(f)\cap Y_l\hookrightarrow   (l+\ker( f))\cap Y_l\to \qmod{ (l+\ker( f))\cap Y_l}{l}
$$
is surjective because $l\subset Y_l$, and it  is injective since $l\cap \ker(f)=\{0\}$. Therefore one has $f^{-1}(f(l))\cap Y_l=l$ if and only if  $\ker(f)\cap Y_{l}=\{0\}$.
 
\end{proof}

In the special case $X\cap  \P(\ker(f))=\emptyset$ we obtain a well-defined smooth map $[f]_X:X\to\P(W)$ between compact manifolds of the same dimension. The main goal of this section is to study the relative orientability of such a map and, in the relatively orientable case,  to study the possible    degrees of $[f]_X$      for a fixed submanifold $X$.\\

Since the projection   $[f]_X$ is defined on all of $X$ if and only if $f$ belongs to the open subset
$$\Hom(V,W)_X:=\{f\in \Hom(V,W)|\ X\cap  \P(\ker(f))=\emptyset\}\ ,
$$
it is natural to study the complement of this open set, namely the {\it wall} associated with $X$.
\begin{dt} \label{wall-def} Let $X\subset \P(V)$ be a compact $m$-dimensional submanifold and  let $W$ be  a real vector space of dimension $n:=m+1$. The wall associated with $X$ is defined by
$${\cal W}_X:=\{f\in \Hom(V,W)|\ X\cap  \P(\ker(f))\ne\emptyset\}\ .
$$
A point $f\in {\cal W}_X$ is  called regular if the following conditions are satisfied:
\begin{enumerate}
\item $\dim(\ker(f))=N-n$ or, equivalently, $f$ is an epimorphism,
\item The intersection $X\cap \P(\ker( f))$ has only one point,   denoted by $\xi_f$, 
\item  $\ker(f)\cap Y_{\xi_f}=\xi_f$.
\end{enumerate}
Denote by $\mathcal{W}_X^0$ the subspace of regular points of the wall.

\end{dt}

The following proposition shows that the wall $\mathcal{W}_X\subset \Hom(V,W)$ is a smooth hypersurface at any regular point $f_0$ and identifies the normal line of  this hypersurface  at $f_0$ canonically with a line depending only on the triple $(X,W,\xi_{f_0})$.

\begin{pr} \label{N} Let $X\subset\P(V)$ be a smooth $m$-dimensional submanifold. Then:
\begin{enumerate}
\item    The wall $\mathcal{W}_X\subset \Hom(V,W)$ is a  smooth  hypersurface at any regular  point $f_0\in \mathcal{W}_X$.
\item Denoting $X\cap\P(\ker(f_0))=\{\xi_0\}$  we have a canonical isomorphism 
$$\psi_{f_0}:N_{\mathcal{W}_X^0,f_0}\to \Hom(\det (Y_{\xi_0}),\det(W))\ .$$
\end{enumerate} 
\end{pr}

\begin{proof}  

1. Consider the incidence varieties 
$$\mathcal{J}:=\{(f,\xi,K)\in \Hom(V,W)\times \P(V)\times G_{N-n}(V)|\ \xi\subset K\subset\ker (f)\}\ .
$$ 
$$\mathcal{J}_X:=\{(f,\xi,K)\in \Hom(V,W)\times X\times G_{N-n}(V)|\ \xi\subset K\subset\ker (f)\}\ .
$$
It's easy to see that $\mathcal{J}$ (respectively $\mathcal{J}_X$) is a vector bundle over a submanifold  of $\P(V)\times G_{N-n}(V)$ (respectively $X\times G_{N-n}(V)$), so it has a natural manifold structure.  Let 
   $q:\mathcal{J}\to \Hom(V,W)$,  $q_X:\mathcal{J}_X\to \Hom(V,W)$   be the projections  on the first factor, which are obviously proper smooth maps.

Note that
\begin{equation}
\mathcal{W}_X=q_X(\mathcal{J}_X)\ .
\end{equation}

One has a canonical identification
$$T_{(f,\xi,K)}(\mathcal{\mathcal{J}})=\{(\varphi,\alpha,\beta)\in\Hom(V,W)\times\Hom(\xi,V/\xi)\times\Hom(K,V/K)|$$
$$\ \resto{  \varphi}{K}=f\circ\beta,\ \resto{\beta}{\xi}=\bar\alpha\}\ ,
$$
where $\bar\alpha$ denotes the composition of $\alpha$ with the natural epimorphism $V/\xi\to V/K$.  Via this identification we have 
$q_{*,{(f,\xi,K)}}(\varphi,\alpha,\beta)=\varphi$, hence 
$$\ker(q_{*,{(f,\xi,K)}})=\{(\alpha,\beta)\in \Hom(\xi,V/\xi)\times\Hom(K,\ker(f)/K)|\  \resto{\beta}{\xi}=\bar\alpha\}\ .
$$
Similarly
$$\ker((q_X)_{*,{(f,\xi,K)}})=\{(\alpha,\beta)\in \Hom(\xi,Y_\xi/\xi)\times\Hom(K,\ker(f)/K)|\  \resto{\beta}{\xi}=\bar\alpha\}\ .
$$

Let now $f_0\in \mathcal{W}_X^0$ be a regular point on the wall, and put  $K_0:= \ker(f_0)$, $\xi_0:=\xi_{f_0}$. We will show that $q_X$ is an immersion at $(f_0,\xi_0,K_0)$. Indeed,  since we have $K_0=\ker (f_0)$, the kernel $\ker((q_X)_{*,{(f_0,\xi_0,K_0)}})$ can be identified with 
$$\{\alpha\in \Hom(\xi_0,Y_{\xi_0}/\xi_0)|\ \bar\alpha=0\}=\Hom(\xi_0,(K_0\cap Y_{\xi_0})/\xi_0)\ ,$$
which vanishes because $K_0\cap Y_{\xi_0}=\xi_0$. Therefore $q_X$ is an immersion at $(f_0,\xi_0,K_0)$ as claimed. Since the fibre  $q_X^{-1}(f_0)$ has  only one element and  $q_X$ is  proper, it follows that $q_X(U)$ is a neighborhood of $f_0$  for every neighborhood $U$ of $(f_0,\xi_0,K_0)$ in  $\mathcal{J}_X$. This implies that the image $\mathcal{W}_X=q_X(\mathcal{J}_X)$ is a submanifold of $\Hom(V,W)$ at $f_0$ whose germ at $f_0$ can be identified with the germ at  $(f_0,\xi_0,K_0)$ of $\mathcal{J}_X$. A simple dimension count shows that $\mathcal{W}_X$ is a hypersurface at $f_0$.
\\ \\
2. For the second statement of the proposition  we will construct a canonical isomorphism 
$$\psi_{f_0}:N_{\mathcal{W}_X^0,f_0}\to \Hom(\det (Y_{\xi_0}),\det(W))$$
 as the composition $u_{f_0}\circ a_{f_0}$ of two canonical isomorphisms:

$$a_{f_0}:N_{\mathcal{W}_X^0,f_0}\to \Hom\left(\xi_0,\qmod{W}{f_0(Y_{\xi_0})}\right)\  ,\ $$
$$u_{f_0}: \Hom\left(\xi_0,\qmod{W}{f_0(Y_{\xi_0})}\right)\to  \Hom(\det (Y_{\xi_0}),\det(W))\ .$$
We define first
$$A_{f_0}:\Hom(V,W)=T_{f_0}\Hom(V,W)\to  \Hom\left(\xi_0,\qmod{W}{f_0(Y_{\xi_0})}\right) 
$$
by
\begin{equation}
\label{alpha}A_{f_0}(\varphi):=\resto{\varphi}{\xi_0}\ \mathrm{mod}\    f_0(Y_{\xi_0})\ .
\end{equation}
It is easy to see that $A_{f_0}(\varphi)=0$ if and only if there exists  %
$$(\alpha,\beta)\in \Hom(\xi_0,Y_{\xi_0}/\xi_0)\times \Hom(K_0,V/K_0)$$
such that $(\varphi,\alpha,\beta)\in T_{(f_0,\xi_0,K_0)}(\mathcal{J}_X)$, i.e., if and only if $\varphi\in T_{f_0}({\cal W}_X^0)$.  Hence $A_{f_0}$ induces a canonical isomorphism $a_{f_0}:N_{\mathcal{W}_X^0,f_0}\to \Hom\left(\xi_0, {W}{/f_0(Y_{\xi_0})}\right)$ as claimed.\\

For the construction of $u_{f_0}$ we use  the exact sequences: 
$$0\to \xi_0\hookrightarrow   Y_{\xi_0}\to \qmod{ Y_{\xi_0}}{\xi_0}\to 0\ ,
$$

$$0\to f_0(Y_{\xi_0})\hookrightarrow W\to \qmod{W}{ f_0(Y_{\xi_0})}\to 0\ ,
$$
which give  standard isomorphisms
\begin{equation}\label{isos} \det (Y_{\xi_0})=\xi_0\otimes\det \big(\qmod{ Y_{\xi_0}}{\xi_0}\big)\ ,\ 
\end{equation}
\begin{equation}\label{standard}
\det(W)=\det(f_0(Y_{\xi_0}))\otimes \qmod{W}{ f_0(Y_{\xi_0})}\
\end{equation}

Note that one has a second canonical (but non-standard!)  isomorphism
\begin{equation}\label{isosNS}
\det(W)=\qmod{W}{ f_0(Y_{\xi_0})} \otimes\det(f_0(Y_{\xi_0})) 
\end{equation}
defined by  $[w]\otimes \delta\mapsto w\wedge\delta$ (see Remark \ref{det-rem} below), which is more convenient in our situation, because the maps we consider  here relate the factors  of the tensor products in (\ref{isos}), (\ref{isosNS}) respecting the order.

The isomorphism $u_{f_0}: \Hom\left(\xi_0,\qmod{W}{f_0(Y_{\xi_0})}\right)\to  \Hom(\det (Y_{\xi_0}),\det(W))$ is defined {\it via the isomorphisms (\ref{isos}) and  (\ref{isosNS})} by
\begin{equation}\label{u} u_{f_0}(\sigma)=\sigma\otimes \det(\bar f_0)\ ,
\end{equation}
where $\bar f_0:\qmod{ Y_{\xi_0}}{\xi_0}\to f_0(Y_{\xi_0})$ is the isomorphism induced by $f_0$.

Note that if one uses the standard isomorphism (\ref{standard}) for $\det(W)$  the corresponding formula for $u_{f_0}$ would be
$$u_{f_0}(\sigma)(v\otimes \delta):=(-1)^m(\sigma(v)\otimes[\det(\bar f_0)(\delta)]\ .
$$
 
\end{proof}

\begin{re} \label{det-rem} Let $C$ be  a finite dimensional real vector space. For every subspace  $A\subset C$ we define the isomorphisms
$$u_A:\det(A)\otimes\det(C/A)\textmap{\simeq} \det(C),\ v_A:\det(C/A)\otimes\det(A)\textmap{\simeq}\det(C)
$$
by 
$$u_A(\delta\otimes[c])=\delta\wedge c\ ,\ v_A([c]\otimes \delta)= c\wedge \delta\ .
$$
\begin{enumerate}
\item Let 
$$\pg:
\det(A)\otimes\det(C/A)\to \det(C/A)\otimes \det(A)$$
be the obvious isomorphism defined by  permutation of the factors, and put  $a:=\dim(A)$,  $c:=\dim(C)$. Then  one has 
$$v_A^{-1}\circ u_A=(-1)^{a(c-a)} \pg\ .$$
\item
Suppose $C$ has an internal direct sum decomposition  $C=A\oplus B$, and let  $\alpha:A\textmap{\simeq} C/B$, $\beta:B\textmap{\simeq} C/A$ be the obvious isomorphisms. Then one has
$$\det (\alpha)\otimes\det(\beta)^{-1}=v_B^{-1}\circ u_A \ .$$
In other words,  via the isomorphisms   $u_A$, $v_B$ the tensor product $\det (\alpha)\otimes\det(\beta)^{-1}$ induces $\id_{\det(C)}$.

\end{enumerate}
\end{re}

\subsection{Relative orientations of central projections}
Our next goal is  to describe explicitly the set of relative orientations of a map $[f]_X$ associated with a morphism $f\in\Hom(V,W)_X$.  We will obtain an important (and surprising) result: relative orientability of a map $[f]_X$ depends only on the embedding $X\subset \P(V)$,  and the set of of relative orientations of  $[f]_X$ can be canonically identified with a set  which is intrinsically associated with this embedding, and  is independent of the choice of $f\in\Hom(V,W)$. First  we give  a simple description of the line bundle $[f]^*(\det(T_{\P(W)}))$ on $\P(V)\setminus\P(\ker(f))$:
 \\

Let again $\lambda_V$ (respectively $\lambda_W$) be the tautological line bundle on $\P(V)$ (respectively $\P(W)$). The  family of isomorphisms 
$$(r_{f,l}:l\to f(l))_{l\in\P(V)\setminus\P(\ker(f)}$$
 defines a line bundle isomorphism 
$$r_f:\resto{\lambda_V}{\P(V)\setminus\P(\ker(f))} \textmap{\simeq} [f]^*(\lambda_W) \ .$$
Using the canonical isomorphism
$$\det(T_{\P(W)})=[\lambda_W^\vee]^{\otimes n}\otimes \det(W)\ ,
$$
we obtain a canonical isomorphism
\begin{equation}\label{canonicalIso}[f]^*(\det(T_{\P(W)}))
% [f]^*[\lambda_W^\vee]^{\otimes n}\otimes \det(W)
\textmap{{(r_f^{\vee})^{\otimes n}\otimes\id}} [\resto{\lambda_V^\vee}{\P(V)\setminus\P(\ker(f))}]^{\otimes n}\otimes \det(W) \ .  \end{equation}
We can prove now
\begin{lm}\label{pullBack} Let  $f\in\Hom(V,W)_X$ and  let  $[f]_X:X\to\P(W)$ be the projection induced by $f$. Then:
\begin{enumerate}
\item There is a canonical isomorphism
$$
 [f]_X^*(\det(T_{\P(W)}))= [ \lambda_{V,X}^\vee]^{\otimes n}\otimes \det(W)\ .
$$
\item The isomorphism class  of the line bundle $ [f]_X^*(\det(T_{\P(W)}))$ on $X$  is independent of $f\in\Hom(V,W)_X$.
\item The restriction $[f]_X:X\to\P(W)$ is relatively orientable if and only if $w_1(X)=n\big\{\resto{w_1(\lambda_V)}{X}\big\}$.
\item The data of an isomorphism 
$$\nu:[f]_X^*(\det(T_{\P(W)}))\to  \det(T_X)$$
is equivalent to the data of a line bundle isomorphism 
$$\mu:X\times\det(W)\textmap{\simeq} \det(Y_X) ,$$
hence to the data of a global trivialization   of   $\det(Y_X)$ with fibre $\det(W)$.
\end{enumerate}
\end{lm}

\begin{proof} The first statement follows by  restricting the isomorphism (\ref{canonicalIso}) to $X$. The statements (2) and (3) are direct consequences of (1) whereas (4) follows from (1) and the canonical isomorphism
$$\det(T_X)= [ \lambda_{V,X}^\vee]^{\otimes n}\otimes
\det(Y_{X})$$
induced by (\ref{tangentX}).

 \end{proof}

This proves the following important
\begin{pr} \label{indep-f} Relative  orientability of a map $[f]_X:X\to\P(W)$ defined by  $f\in\Hom(V,W)_X$ depends only on the embedding $X\subset\P(V)$, and is independent of  $f$. The set of relative orientations of such a  map $[f]_X:X\to\P(W)$ can be canonically identified with the set of connected components of the space of line bundle isomorphisms 
$$\mu:X\times\det(W)\textmap{\simeq} \det(Y_X) ,$$
hence this set is independent of $f$.
\end{pr}

\begin{dt}\label{nu} A bundle isomorphism $\mu:X\times\det(W)\textmap{\simeq} \det(Y_X) $  will be called orientation parameter for projections $X\to\P(W)$. Given $f\in \Hom(V,W)_X$ we denote by 
$$\nu(f,\mu):=[(r_f^{\vee})^{\otimes n}\otimes \mu]$$
 the relative orientation of $[f]_X$ associated with  the  orientation parameter  $\mu$.
% using Proposition  \ref{indep-f}.
\end{dt}

Therefore for every   orientation parameter $\mu:X\times\det(W)\textmap{\simeq} \det(Y_X) $ we obtain a well defined map
$$\deg_\mu^X:\Hom(V,W)_X=\Hom(V,W)\setminus{\cal W}_X\to \Z\ ,
$$
which obviously descends to a map (denoted by the same symbol):
$$\deg_\mu^X:\pi_0(\Hom(V,W)\setminus{\cal W}_X)\to \Z\
$$

Inspired by  the terminology used in gauge theory we introduce the following 
\begin{dt}\label{chambers} 
The connected components of  $\Hom(V,W)\setminus{\cal W}_X$  will be called chambers. 
\end{dt} 

In the next section we will focus on  the following problem: how can one compare the values of the degree map on different chambers. The first step  will be   a wall-crossing formula which computes the jump of the degree map when one crosses the wall transversally at a regular point. The main ingredient  in proving this wall-crossing formula is the following  remark 
which states  that, for a given point $\xi_0\in X$, the set of isomorphisms  $\det(Y_{\xi_0})\to \det(W)$ has two (radically different) geometric interpretations:

\begin{re}  Fix $\xi_0\in X$. The orientation parameter $\mu$ defines an isomorphism $\mu_{\xi_0}^{-1}:\det(Y_{\xi_0})\to \det(W)$, which (by  Proposition \ref{N})   also defines a generator  of the normal line $N_{{\cal W}^0_X,f}$ for every  $f\in {\cal W}^0_X$ for which  $\xi_f=\xi_0$. 
\end{re}

This important remark will play a crucial role in the following section. \\

Fix now an orientation parameter  $\mu$,   and suppose that $[f]_X$ is well-defined and  a local isomorphism at $\xi_0\in X$. By Remark \ref{localiso} this means that $\ker(f)\cap Y_{\xi_0}=\{0\}$. We need an explicit formula for the local degree $\deg_{\nu(f,\mu),\xi_0}([f]_X)$. Using the exact sequences
$$ 0\to \xi_0\to Y_{\xi_0}\to Y_{\xi_0}/\xi_0\to  0\ ,\ 0\to f(\xi_0)\to W\to W/f(\xi_0)\to 0\ ,
$$
we get   canonical isomorphisms
\begin{equation}\label{newisos}
\det(Y_{\xi_0})=\xi_0\otimes \det(Y_{\xi_0}/\xi_0)\ ,
$$
$$\det(W)=f(\xi_0)\otimes \det (W/f_0(\xi_0))\ .
\end{equation}
By Remark \ref{diffX} the differential 
$$[f]_{*,\xi_0}:T_{\xi_0}(X)=\Hom(\xi_0,Y_{\xi_0}/\xi_0)\to \Hom(f(\xi_0),W/f(\xi_0))=T_{f{(\xi_0)}}\P(W)$$
 is given by
$$[f]_{*,\xi_0}(\varphi)=p_{f,\xi_0}\circ\varphi\circ r_{f,\xi_0}^{-1}\ ,$$
where $r_{f,\xi_0}:\xi_0\to f(\xi_0)$, $p_{f,\xi_0}:Y_{\xi_0}/\xi_0\to W/f(\xi_0)$ are the obvious isomorphisms induced by $f$. Therefore

\begin{lm} \label{KH} Let   $\mu$ be an orientation parameter,   and suppose that $[f]_X$ is well-defined and  a local isomorphism at $\xi_0\in X$. Via the isomorphisms (\ref{newisos}) the local degree $\deg_{\nu(f,\mu),\xi_0}([f]_X)$ is given by
\begin{equation}\label{KHeq} \deg_{\nu(f,\mu),\xi_0}([f]_X)=\mathrm{sign}(\mu_{\xi_0}\circ  (r_{f,\xi_0}\otimes \det(p_{f,\xi_0})))\ .
\end{equation}
\end{lm} 

We will also need the following simple remark concerning the functoriality of the degree map with respect to isomorphisms $\Phi:W\to W'$.
\begin{re} \label{Phi} Let $\mu:X\times\det(W)\to \det(Y_X)$ be an  orientation parameter, $f\in\Hom(V,W)_X$ and $\Phi:W\to W'$ a   vector space isomorphism. Then  
$$\deg^X_{\nu(f,\mu\circ \det(\Phi)^{-1})}([\Phi\circ f]_X)=\deg^X_{\nu(f,\mu)}([ f]_X)\ .
$$
In particular, if $W=W'$ one has:
$$\deg^X_{\nu(f,\mu)}([\Phi\circ f]_X)=\mathrm{sign}(\det(\Phi))\deg^X_{\nu(f,\mu)}([ f]_X)\ .
$$
\end{re}
\vspace{3mm}
We end this section with an example.

{\ }\\
{\bf Example:} ({\it Veronese maps}) Let $W$ be a real vector space of  dimension $n\geq 2$, and let $g:\P(W)\to \P(W)$  be a regular real algebraic map. Such a map factorizes as $g=[f]\circ v_d$, where %
$$v_d:\P(W)\to \P(S^d W  )$$ 
is the Veronese map of degree $d$, and $f\in\Hom(S^d(W),W)$   with $\P(\ker(f))\cap\im(v_d)=\emptyset$.  The positive integer $d$ is determined by $g$  and will be called the algebraic degree of $g$.
Applying Lemma \ref{pullBack} to the image $X:=v_d(\P(W))$ and noting that $v_d^*(\lambda_{S^d(W)})= \lambda_W^{\otimes d}$, one obtains:
$$\det(T_{\P(W)})^\vee\otimes g^*(\det(T_{\P(W)}))=\lambda_W^{\otimes[n(1-d)]}
$$
This  proves  the following simple, but interesting, result:
\begin{pr} A regular real algebraic map $g:\P(W)\to \P(W)$ of algebraic degree $d$ is relatively orientable if and only if $\dim(W)(1-d)$ is even, and in this case it is canonically relatively oriented.
\end{pr}

In the case when   $\dim(W)(1-d)$ is even, it is an interesting problem to determine the possible degrees of such regular real algebraic maps with respect to this canonical relative orientation. The case $\dim(W)=2$ is known  by the work of Brockett   (see \cite{Se} p. 41  or  \cite{By} Theorem 2.1) and will be described in detail at the end of section \ref{wall-cross-section}.

\vspace{3mm} 

\subsection{Wall crossing jump} \label{wall-cross-section}

 Suppose now that the  condition in the third statement of  Lemma \ref{pullBack} is satisfied, and fix  an isomorphism 
$$\mu:X\times\det(W)\to \det(Y_X) .$$
We are interested in the map $\deg_{\mu}^X:\pi_0(\Hom(V,W)_X)\to\Z$ defined by
$$\deg_\mu^X(\hat f):=\deg_{\nu(f,\mu)}([f]_X)\ .
$$
Note that $[f]$ can be written as the composition
$$X\hookrightarrow \P(V)\setminus\P(\ker(f))\to \P(\im(f))\hookrightarrow\P(W)\ ,$$
where the central map is   the projection of $\P(W)$ onto $\P(\im(f))$ with center $\P(\ker(f))$. In other words, we are interested in the variation of the degree of  central  projections   $ \P(V)\supset X\to \P(W)$  when the projection varies. The first step is to determine the variation of 
$\deg_{\nu(f,\mu)}([f]_X)$ as $f$ varies on a curve which crosses the wall transversally at a regular point.

\begin{thry} \label{wall-cross-th} (Wall-crossing formula) Let  $f_0\in \mathcal{W}_X^0$  be a regular  point on the wall, put $\xi_0:=\xi_{f_0}$,  $N_0:= N_{{\cal W}_X^0,f_0}$, and let   
$$\fg=(f_\tau)_{\tau\in N_{0}}:N_{0}\to\Hom(V,W)$$
be   a smooth map such that the differential $\fg_{*,0}$ is a right splitting of the exact sequence
$$0\to  T_{\mathcal{W}_X^0,f_0}\to  T_{\Hom(V,W),f_0}=\Hom(V,W)\to N_{0}\to 0\ .$$
Then for every sufficiently small $\tau\in N_{0}\setminus\{0\}$ we have: 
\begin{enumerate}
\item $f_\tau\in\Hom(V,W)_X$ and  $[f_\tau]_X$ is a local diffeomorphism at $\xi_0$,
\item  $\deg_{\nu(f_\tau,\mu),\xi_0}([f_\tau]_X)=\mathrm{sign}(\mu_{\xi_0}\circ(\psi_{f_0}(\tau)))$, 
\item $\deg_{\nu(f_\tau,\mu)}([f_\tau]_X)-\deg_{\nu(f_{-\tau},\mu)}([f_{-\tau}]_X)=2\mathrm{sign}(\mu_{\xi_0}\circ(\psi_{f_0}(\tau)))$.
\end{enumerate}
 \end{thry}

\begin{proof}

 (1) The fact that $f_\tau\in\Hom(V,W)_X$  for sufficiently small $\tau\in N_{0}\setminus\{0\}$ follows directly from Proposition \ref{N} taking into account that  the complement ${\cal W}_X=\Hom(V,W)\setminus \Hom(V,W)_X$ is a smooth hypersurface at $f_0$ and the curve $(f_\tau)_{\tau\in N_{0}}$ is transversal to this hypersurface at $f_0$. In order to prove that $[f_\tau]_X$ is a local diffeomorphism at $\xi_0$ we use Remark \ref{localiso}.  We have to show that $\ker(f_\tau)\cap Y_{\xi_0}=\{0\}$.   

Note that, in general, for two finite dimensional real vector spaces  $A$, $B$   of the same dimension, the closed subset
$${\cal W}(A,B):=\{s\in \Hom(A,B)|\ \ker(s)\ne\{0\}\}
$$ 
of $\Hom(A,B)$ is a  smooth hypersurface at any point $s_0$ with $\dim(\ker(s_0))=1$, and the tangent space at such a point is
$$T_{s_0}{\cal W}(A,B)=\{\sigma\in \Hom(A,B)|\ \sigma(\ker(s_0))\subset s_0(A)\}\ .
$$
Therefore  a tangent vector  $\sigma \in  T_{s_0}(\Hom(A,B))$ is transversal to $ {\cal W}(A,B)$ at $s_0$ if and only if the linear map $\ker(s_0)\to B/s_0(A))$ induced by $\resto{\sigma}{\ker(s_0)}$ is an isomorphism.

Using this remark we see that the map $\tilde\fg=(\tilde f_\tau)_\tau:N_0\to \Hom(Y_{\xi_0},W)$ given by $\tilde f_\tau=\resto{f_\tau}{Y_{\xi_0}}$ is transversal to ${\cal W}(Y_{\xi_0},W)$ at $f_0$. Indeed, using the notations introduced in the proof of Proposition \ref{N}  the map 
$$\ker \tilde  f_0=\xi_0\longrightarrow  W/f_0(Y_{\xi_0}) $$
 induced by $\tilde\fg_{*,0}(\tau)$ is precisely $A_{f_0}(\fg_{*,0}(\tau))$  by  the definition of $A_{f_0}$. On the other hand, since $\fg_{*,0} $ is a right inverse of the canonical projection $\Hom(V,W)\to N_{0}$, we see that $A_{f_0}(\fg_{*,0}(\tau))=a_{f_0}(\tau)$, which is nonzero for $\tau\in N_{0}\setminus\{0\}$ because $a_{f_0}$ is an isomorphism by Proposition \ref{N}.
\\ \\
(2)  
We  suppose  first that $\fg$ is an affine map,  so it has the form
\begin{equation}\label{affine}
f_\tau=f_0+\phi(\tau)\ ,
\end{equation}
where $\phi:N_0\to \Hom(V,W)$ is a linear map (which coincides with the differential $\fg_{*,0}$).  
As we have seen in the proof of (1) our assumption about the differential $\fg_{*,0}$ implies that $\phi(\tau)$ is a lift of $\tau$, so that   $a_{f_0}(\tau)=A_{f_0}(\phi(\tau))$. Therefore  for  $\tau\ne 0$ the morphism 
$$a_{f_0}(\tau)=A_{f_0}(\phi(\tau))=\resto{\bar\phi(\tau)}{\xi_0}:\xi_0\to W/f_0(Y_{\xi_0}) $$
 induced by $\phi(\tau)$ is an isomorphism. For every 
$\tau\in N_0\setminus\{0\}$ we obtain a direct sum decomposition
\begin{equation}\label{direct} W=\phi(\tau)(\xi_0)\oplus f_0(Y_{\xi_0})\ ,
\end{equation}
which is  independent  of $\tau$, since $N_0$ is 1-dimensional.
Now fix $\tau\in N_0\setminus\{0\}$. We have an obvious  commutative diagram with exact rows
\begin{equation}\label{diagram}
\begin{array}{c}
\unitlength=1mm
\begin{picture}(100,30)(-37,-19)
\put(-40,4){0}
\put(-34,5){\vector(2,0){12}}
\put(-18,4){$\xi_0$}
\put(-11,5){\vector(2,0){12}}
\put(4,4){$Y_{\xi_0}$}
\put(11,5){\vector(2,0){8}}
\put(23,4){$\qmod{Y_{\xi_0}}{\xi_0} $}
\put(37,5){\vector(2,0){15}}
\put(56,4){0}
\put(-16,0){\vector(0, -3){9}}
\put(6,0){\vector(0, -3){9}}
\put(27,0){\vector(0, -3){9}}

\put(-15,-5){$r_{\phi(\tau),\xi_0}$}
\put(7,-5){$\resto{f_{\tau}}{Y_{\xi_0}}$}
\put(28,-5){$p_{0}+\phi(\tau)_0$}

\put(-40,-16){0}
\put(-34,-15){\vector(2,0){12}}
\put(-18,-16){$\xi_0$}
\put(-11,-15){\vector(2,0){12}}
\put(4,-16){$W$}
\put(11,-15){\vector(2,0){8}}
\put(22,-16){$\qmod{W}{\phi(\tau)(\xi_0)} $}
\put(42,-15){\vector(2,0){10}}
\put(56,-16){0}

\put(60,-16){,}
\end{picture} 
\end{array} 
\end{equation}
where 
$$p_0:{Y_{\xi_0}}/{\xi_0}\to \qmod{W}{\phi(\tau)(\xi_0)}\ ,\ \phi(\tau)_0:{Y_{\xi_0}}/{\xi_0}\to \qmod{W}{\phi(\tau)(\xi_0)}$$
 are the linear maps induced by  $f_0$, and  $\phi(\tau)$  respectively.   Note that 
$p_0$ is an isomorphism, because it can be written as a composition  of isomorphisms:
$${Y_{\xi_0}}/{\xi_0}\to f_0(Y_{\xi_0})\to   {W}/{\phi(\tau)(\xi_0)}\ .$$

The diagram (\ref{diagram}) shows that
\begin{enumerate}[(a)]
\item  \label{b} $r_{f_{\tau},\xi_0}=r_{\phi(\tau),\xi_0}$  for every  sufficiently small $\tau\in N_0\setminus\{0\}$, 
\item    \label{c} $p_{f_{\tau},\xi_0}=p_0+\phi(\tau)_0$ for every    sufficiently small $\tau\in N_0\setminus\{0\}$.
\end{enumerate}

Using  formula (\ref{KHeq}) of Lemma \ref{KH} and taking into account  (\ref{b}),  (\ref{c}) we see that, via the canonical  isomorphisms
$$\det(Y_{\xi_0})=\xi_0\otimes \det(Y_{\xi_0}/\xi_0)\ ,\ $$
$$\det(W)=\phi(\tau)(\xi_0)\otimes \det(W/\phi(\tau)(\xi_0))\ ,
$$
we have  for every  sufficiently small $\tau\in N_0\setminus\{0\}$
\begin{equation}\label{degdef} \deg_{\nu(f_{\tau},\mu),\xi_0}([f_{\tau}]_X)=\mathrm{sign}(\mu_{\xi_0}\circ (r_{\phi(\tau),\xi_0}\otimes\det(p_{0}+\phi(\tau)_0 )))$$
$$=\mathrm{sign}(\mu_{\xi_0}\circ (r_{\phi(\tau),\xi_0}\otimes\det(p_{0} )))\ .
\end{equation}
For the last equality we used  $\lim_{t\to 0} (p_0+t\phi(\tau)_0)=p_0$ and the continuity of the determinant. Now we use  the canonical isomorphism
$$v:\qmod{W}{f_0(Y_{\xi_0})}\otimes  \det ( f_0(Y_{\xi_0}))\to \det(W)
$$
and we apply    Remark \ref{det-rem} to the subspaces $A:=\phi(\tau)(\xi_0)$ (for $\tau\ne 0$), $B:= f_0(Y_{\xi_0})$ of $W$. Therefore  let
$$\alpha:\phi(\tau)(\xi_0)\to \qmod{W}{f_0(Y_{\xi_0})}\ ,\ \beta:   f_0(Y_{\xi_0})\to \qmod{W}{\phi(\tau)(\xi_0)} 
$$
be the isomorphisms associated with the direct sum decomposition (\ref{direct}). Using the second statement of  Remark \ref{det-rem}  we see that the equality (\ref{degdef}) remains true if we replace $r_{\varphi_0,\xi_0}$ by  
$$\alpha\circ   r_{\phi(\tau),\xi_0}:\xi_0\to  \qmod{W}{f_0(Y_{\xi_0})}\ ,$$
and $p_0$ by 
$$\beta^{-1}\circ p_0:Y_{\xi_0}/\xi_0\to f_0(Y_{\xi_0})\ .$$
But 
$$\alpha\circ   r_{\phi(\tau),\xi_0}=\resto{\bar\phi(\tau)}{\xi_0}:\xi_0\to Y_{\xi_0}/\xi_0\ ,\ \beta^{-1}\circ p_0=\bar f_0:Y_{\xi_0}/\xi_0\to f_0(Y_{\xi_0})\ .$$
Therefore
\begin{equation}\label{degnew} \deg_{\nu(f_{\tau},\mu),\xi_0}([f_{\tau}]_X)=\mathrm{sign}(\mu_{\xi_0}\circ (\resto{\bar\phi(\tau)}{\xi_0}\otimes\det(\bar f_0 )))\ \forall \tau\in N_0\setminus\{0\}\ .
\end{equation}

Now  recall that $\resto{\bar\phi(\tau)}{\xi_0}=a_{f_0}(\tau)$ and that 
$$a_{f_0}(\tau)\otimes\det(\bar f_0 )=u_{f_0}(a_{f_0}(\tau))=\psi_{f_0}(\tau)$$
 by the definitions  of $u_{f_0}$ and $\psi_{f_0}$. This proves the claim in the  case of an affine map $\fg$.\\

In order to prove the statement for  a general map   $\fg$ note that the space  ${\cal F}_{f_0}$ of maps $\fg:N_0\to\Hom(V,W)$ satisfying the hypotheses of the theorem is a closed affine subspace of the Fréchet space ${\cal C}^\infty(N_0,\Hom(V,W))$.  Fix a norm on the line $N_0$. For a bounded (with respect to the ${\cal C}^\infty$-topology) subset ${\cal K}\subset {\cal F}_{f_0}$  we can find $\varepsilon>0$ such that $[f_\tau]_X$  is defined and is a local diffeomorphism at $\xi_0$ for every $\fg\in {\cal K}$  and  every  $\tau\in N_0\setminus\{0\}$ with $\|\tau\|<\varepsilon$. This shows that the map 
$$\fg\mapsto \deg_{\nu(f_{ \tau},\mu),\xi_0}([f_{\tau}]_X)\ \hbox{ for small } \tau\in N_0\setminus\{0\}$$
is locally constant on ${\cal F}_{f_0}$. But this space is connected and contains affine maps.
  \\ \\
(3)   Note first  that it is sufficient to prove the claimed formula for a special  map $\fg:N_0\to \Hom(V,W)$ satisfying the hypothesis of the theorem. This is the case since, for two such maps $\fg$, $\g$, the points $\fg(\tau)$, $\g(\tau)$ belong to the same chamber  (see Definition \ref{chambers}) for any sufficiently small $\tau\in N_0\setminus\{0\}$.\\

We will construct a special map $\fg:N_0\to \Hom(V,W)$ which satisfies the hypothesis of the theorem, is affine, and has the following remarkable property:
\\ \\
{\bf P.}  {\it There exists $\zeta_0\in\P(W)$ such that the  subspace $L_0:=f_{\tau}^{-1}(\zeta_0)$   is independent of $\tau\in N_0$ and $\P(L_0)$ is transversal to  $X$ at any intersection point.
}
\\ 

The existence of such a map solves our problem. Indeed,  since $X$ is compact and $\P(L_0)$ is transversal to  $X$ at any intersection point,  the intersection $F:=\P(L_0)\cap X$ is finite.  By (1) we know that, for every sufficiently small $\tau\in N_0\setminus\{0\}$,  the map $[f_\tau]_X$ is defined on all of $X$. On the other hand our transversality condition implies that  $\zeta_0$ is a regular value of these maps.  

 Taking $\tau=0$ in the first condition of {\bf P}    we see that $\xi_0\in F$. Applying Proposition \ref{degformula} to $[f_\tau]_X$ for sufficiently small   $\tau\in N_0\setminus\{0\}$, we obtain
$$\deg_{\nu(f_\tau,\mu)}([f_\tau]_X)=\deg_{\nu(f_\tau,\mu),\xi_0}([f_\tau]_X)
+\sum_{\xi\in  F\setminus\{\xi_0\}} \deg_{\nu(f_\tau,\mu),\xi}([f_\tau]_X)\ .
$$
Since the second term is obviously independent  of $\tau\ne 0$   we get
$$\deg_{\nu(f_\tau,\mu)}([f_\tau]_X)-\deg_{\nu(f_{-\tau},\mu)}([f_{-\tau}]_X)=\deg_{\nu(f_\tau,\mu),\xi_0}([f_\tau]_X)-\deg_{\nu(f_{-\tau},\mu),\xi_0}([f_{-\tau}]_X)\ ,
$$
so the result follows from (2).
\\

 We conclude the proof of the theorem with the construction of an affine  map $\fg=(f_\tau)_{\tau\in N_0}$ which satisfies the hypothesis of the theorem  and has the property {\bf P}.

Put  $K_0:=\ker(f_0)$. Since $f_0$ is an epimorphism, we have $\dim(K_0)=N-n$ and the space ${\cal L}_{K_0}$ of  $(N-n+1)$-dimensional  linear subspaces   $L\subset V$ with $K_0\subset L$ can be identified with $\P(W)$ via $f_0$. The subset 
$${\cal L}_0:=\{L\in {\cal L}_{K_0}|\ L\cap Y_{\xi_0}=\xi_0\}$$
is non-empty and Zariski open, hence open and dense in  ${\cal L}_{K_0}$. Let $L_0\in {\cal L}_0$ an element which corresponds to a regular value $\zeta_0$ of the projection 
$$[f_0]_X:X\setminus\{\xi_0\}\to\P(W)\ .$$ The existence of such a point follows  from Sard's theorem.

Note that  $\P(L_0)$ is transversal to $X$ at any intersection point $\xi\in\P(L_0)\cap X$. Indeed, the condition $L_0\cap Y_{\xi_0}=\xi_0$ implies that $\P(L_0)$ is transversal to $X$ at $\xi_0$, whereas the condition that $\zeta_0=f_0(L_0)$ is a regular value of $[f_0]_X:X\setminus\{\xi_0\}\to\P(W)$ implies that $\P(L_0)$ is transversal to $X$ at any point $\xi\in \P(L_0)\cap X\setminus\{\xi_0\}$.\\

Therefore the obtained  $(N-n+1)$-dimensional subspace $L_0$ has the properties:
\begin{enumerate}[(a)]
\item \label{contiansxi} $K_0\subset L_0$,
\item \label{IntTrans}   $L_0\cap Y_{\xi_0}=\xi_0$,
\item   \label{IntWithX} $\P(L_0)$ is transversal to $X$ at any intersection point.\end{enumerate}

Now we choose a complement for each  of the three inclusions in the chain
$$\xi_0\subset K_0\subset L_0\subset V\ .
$$
Let $U_0$ be an arbitrary complement of $\xi_0$ in $K_0$, $l_0$ an arbitrary complement of $K_0$ in $L_0$, and $M_0$ a complement of $L_0$ in $V$ which is contained in $Y_{\xi_0}$.   The latter complement exists, because by  (\ref{IntTrans})  any complement of $\xi_0$ in $Y_{\xi_0}$ is also a complement of $L_0$ in $V$.  

We have $\dim(U_0)=N-n-1$, $\dim(l_0)=1$, $\dim(M_0)=n-1$. The sum $l_0+M_0$ is a complement of $K_0$ in $V$, hence $f_0$ induces an isomorphism $l_0+M_0\to W$. We obtain an  induced internal direct sum decomposition of $W=\zeta_0\oplus W_0$ with
$$ \zeta_0:=f_0(l_0)=f_0(L_0)\ ,\ W_0:=f_0(M_0)=f_0(Y_{\xi_0})\ .$$

With respect to the internal direct sum decompositions
$$V=\xi_0\oplus U_0 \oplus l_0\oplus M_0\ ,\ W=\zeta_0\oplus W_0
$$
the map $f_0$ is given by a matrix of the form
$$
\left(\begin{matrix}
0&0&g_0&0\\
0&0&0&h_0
\end{matrix}\right)\ ,
$$
where $g_0:l_0\to \zeta_0$, $h_0:M_0\to W_0$ are the isomorphisms induced by $f_0$. 
We denote by  $r_\tau:\xi_0\to \zeta_0$  the morphism defined as the image of $\tau$ under the composition
$$N_0\textmap{a_{f_0}\simeq}\Hom(\xi_0,W/f_0(Y_{\xi_0}))=\Hom(\xi_0,W/W_0)\stackrel{\simeq}{\longrightarrow}\Hom(\xi_0,\zeta_0)\ ,
$$
and we define
$$f_\tau:=\left(\begin{matrix}
r_\tau&0&g_0&0\\
0&0&0&h_0
\end{matrix}\right)\ .
$$ 
The map $\fg=(f_\tau)_\tau:N_0\to\Hom(V,W)$ is  affine and satisfies the hypothesis  of the theorem (because   $A_{f_0}\circ\fg_{*,0}$ coincides with $a_{f_0}$ by definition of $r_\tau$).
Moreover, since $h_0$ is an isomorphism, for every $\tau\in N_0$ the subspace $f_\tau^{-1}(\zeta_0)$ coincides with $L_0$. Taking into account (\ref{c}) we see that 
 $\fg$ satisfies property ${\bf P}$, which concludes the proof.  
\begin{center}
\includegraphics[scale=0.5]{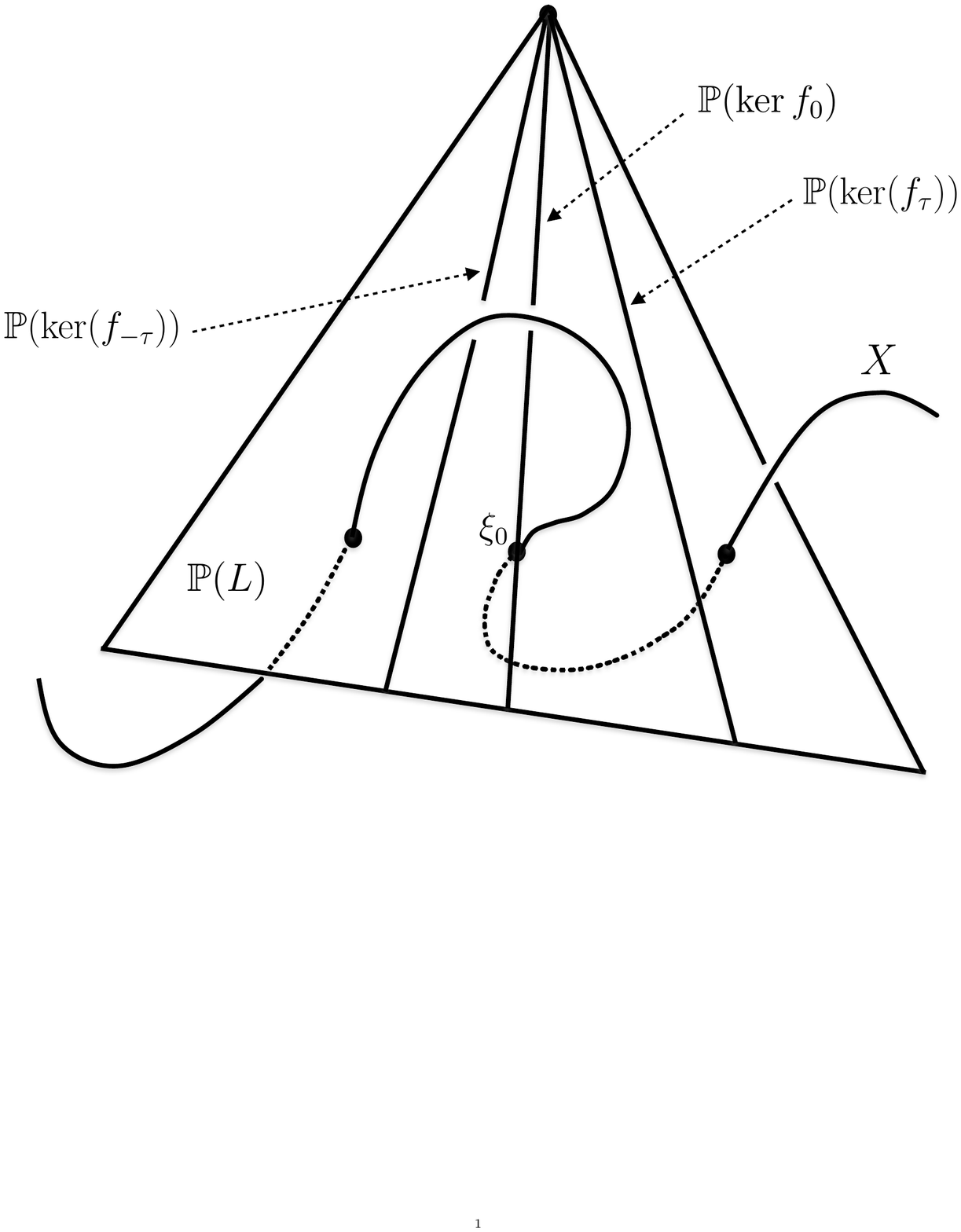}
\end{center}
\vspace{3mm}
 \end{proof}

 \begin{co} \label{wall-cross-co} 
  Let  $f_0\in \mathcal{W}_X^0$  be a regular  point on the wall and let
$$\gamma=(g_t)_t:(-\varepsilon,\varepsilon)\to\Hom(V,W)$$
 be a smooth path such that $\gamma(0)=f_0$ and the image $[\dot\gamma(0)]$ of the velocity vector in $N_{{\cal W}^0_X,f_0}$ is non-zero.  
Then for every sufficiently small $t\in (-\varepsilon,\varepsilon)\setminus\{0\}$ we have: 
\begin{enumerate}
\item $g_t\in\Hom(V,W)_X$ and  $[g_t]_X$ is a local diffeomorphism at $\xi_{f_0}$,
\item  $\deg_{\nu(g_t,\mu),\xi_0}([g_t]_X)=\mathrm{sign}(t) \mathrm{sign}(\mu_{\xi_0}\circ(\psi_{f_0}([\dot\gamma(0)])))$, 
\item $\deg_{\nu(g_t,\mu)}([g_t]_X)-\deg_{\nu(g_{-t},\mu)}([g_{-t}]_X)=2\mathrm{sign}(t)\mathrm{sign}(\mu_{\xi_0}\circ(\psi_{f_0}([\dot\gamma(0)])))$.
\end{enumerate}
\end{co}
\begin{proof} We may suppose that $\gamma$ is given by $g_t=f_{t\tau_0}$, where $\tau_0\in N_{{\cal W}^0_X,f_0}\setminus\{0\}$ and $\fg:N_{{\cal W}^0_X,f_0}\to \Hom(V,W)$ is a map satisfying the hypothesis of Theorem \ref{wall-cross-th}.
\end{proof}

This implies the following general difference formula for paths which cross the wall transversally in regular points:
 
\begin{thry} \label{global} (difference formula) Let  
$\gamma=(g_t)_{t}:[0,1]\to \Hom(V,W)$
 be a smooth path such that
\begin{enumerate}
\item  $g_0$, $g_1\in \Hom(V,W)_X$,
\item  $\im(\gamma)$  intersects the wall $\mathcal{W}_X$ only   in regular points,
\item $\gamma$  is transversal to  $\mathcal{W}_X^0$.
\end{enumerate}
Let  $\gamma^{-1}(\mathcal{W}_X)=\{t_1,\dots,t_k\}$. Then one has
$$\deg_{\nu(g_{1},\mu)}([g_{1}]_X)-\deg_{\nu(g_{0},\mu)}([g_{0}]_X)=2\sum_{i=1}^k\mathrm{sign}(\mu_{\xi_{g_{t_i}}}\circ(\psi_{g_{t_i}}([\dot\gamma(t_i)])))\ ,$$
where $[\dot\gamma(t_i)]$ denotes  the projection of the velocity vector $\dot\gamma(t_i)$ to the normal line $N_{{\cal W}^0_X,g_{t_i}}$
\end{thry}
\begin{proof}  Suppose $t_1<\dots<t_k$. The map 
$t\mapsto  \deg_{\nu(g_{t},\mu)}([g_{t}]_X)$
 is well defined and constant on each  of the  intervals 
$$[0,t_1)\ ,\ (t_1,t_{2})\ , \dots,\ (t_{k-1},t_k)\ ,\ (t_k,1]\ .$$
The jumps are given by Corollary \ref{wall-cross-co}.
\end{proof}

We will now prove  that  any two points $f_0$, $f_1\in \Hom(V,W)_X$ can be connected by a path intersecting  the wall transversally in finitely many regular points.  This result, which has important consequences, is based on   the following 

\begin{thry} The irregular locus $\mathcal{B}_X
=\mathcal{W}_X\setminus \mathcal{W}_X^0$ is closed  and the complement $\Hom(V,W)\setminus \mathcal{B}_X$ is connected.
\end{thry}
\begin{proof} We shall identify $\mathcal{B}_X$ with the union of the images of two  smooth proper maps 
$$\hat q_X:\hat{\mathcal{J}}_X \to \Hom(V,W)\ ,\ r_X:\mathcal{D}_X\to\Hom(V,W)$$
of index -2, $n-N-1$ respectively. Then the result  follows  from Lemma 5.7 in \cite{Te}.\\

Denote by $\Delta$ the diagonal of the product $X\times X$ and consider the real blow up
$\widehat{X\times X}_\Delta$ of $X\times X$  along $\Delta$ (see \cite{Wh} section 3, and \cite{Po} section 4  for a similar construction). Set theoretically one has
$$\widehat{X\times X}_\Delta=\big\{(X\times X)\setminus\Delta\big\}\cup\P(N_\Delta)=\big\{(X\times X)\setminus\Delta\big\}\cup\P(T_X)=$$
$$\big\{(X\times X)\setminus\Delta\big\}\cup\P(Y_X/(\resto{\lambda_V}{X}))\ . 
$$

Therefore a point $\zeta\in \P(N_\Delta)$ above a diagonal point $(\xi,\xi)\in\Delta$ defines a plane $\pi_\zeta\subset Y_\xi$ containing $\xi$. We have a natural {\it smooth} map
$$\psi:\widehat{X\times X}_\Delta\to G_2(V)
$$
defined in the following way:
$$
\psi(\zeta):=\left\{\begin{array}{ccc}
\xi+\eta&\rm when &\zeta=(\xi,\eta)\in (X\times X)\setminus\Delta\ ,\\
\pi_\zeta&\rm when& \zeta\in \P(N_\Delta)\ .
\end{array} \right.
$$

We define
$$\hat{\mathcal{J}}_X:=\{(f,K,\zeta)\in\Hom(V,W)\times G_{N-n}(V)\times \widehat{X\times X}_\Delta|\ \psi(\zeta)\subset K\subset\ker(f)\}\ .
$$
This space has a natural structure of a vector bundle of  rank $n^2$  over the incidence variety
$$\hat{\mathcal{I}}_X=\{(\zeta,K)\in \widehat{X\times X}_\Delta\times G_{N-n}(V)|\ \psi(\zeta)\subset K\}\ .
$$
The variety $\hat{\mathcal{I}}_X$ is a locally trivial fibre bundle over $\widehat{X\times X}_\Delta$ with $n(N-n-2)$-dimensional fibre, hence smooth of dimension   $nN-n^2-2$.  This shows that $\dim(\hat{\mathcal{J}}_X)=nN-2$, so the rank of the projection  $\hat q_X:\hat{\mathcal{J}}_X\to \Hom(V,W)$ is -2.

Put now
$$\mathcal{D}_X:=\{(f,L,\xi)\in\Hom(V,W)\times G_{N-n+1}(V)\times X|\ \xi\subset L\subset\ker(f)\}\ .
$$
$\mathcal{D}_X$ has a natural structure of a $(N+1)(n-1)$-dimensional manifold, because it is a rank $(n-1)n $ vector bundle over an $n-1+(n-1)(N-n)$ dimensional basis. Let $r_X:\mathcal{D}_X\to \Hom(V,W)$  be the projection on the first factor.  This is  a smooth proper map of index   $n-N-1\leq -2$. On the other hand, taking into account Definition \ref{wall-def} we see that
$$\mathcal{B}_X=\hat q_X(\hat{\mathcal{J}}_X)\cup r_X(\mathcal{D}_X)\ ,
$$
hence
$$\Hom(V,W)\setminus \mathcal{B}_X=(\Hom(V,W)\setminus \hat q_X(\hat{\mathcal{J}}_X))\setminus r_X(\mathcal{D}_X\setminus r_X^{-1}(\hat q_X(\hat{\mathcal{J}}_X))\ .
$$
Applying now Lemma 5.7 in \cite{Te} to the proper morphisms
$$\hat q_X\ ,\ \resto{r_X}{\mathcal{D}_X\setminus r_X^{-1}(\hat q_X(\hat{\mathcal{J}}_X))}: \mathcal{D}_X\setminus r_X^{-1}(\hat q_X(\hat{\mathcal{J}}_X))\to  (\Hom(V,W)\setminus \hat q_X(\hat{\mathcal{J}}_X))
$$
we see that  the natural maps
$$\pi_0(\Hom(V,W)\setminus \hat q_X(\hat{\mathcal{J}}_X),f_0)\to \pi_0(\Hom(V,W),f_0)\ ,\ $$
$$
\pi_0(\Hom(V,W)\setminus \mathcal{B}_X,f_0)\to \pi_0(\Hom(V,W)\setminus \hat q_X(\hat{\mathcal{J}}_X),f_0)
$$
are bijections, so that $\Hom(V,W)\setminus \mathcal{B}_X$ is connected, as claimed.
\end{proof}
 
 \begin{co} \label{generic-path}  Any pair $(f_0,f_1)\in\Hom(V,W)_X\times \Hom(V,W)_X$ can be connected by a smooth path $\gamma:[0,1]\to\Hom(V,W)$   which intersects the wall $\mathcal{W}_X$ in finitely many regular points, all intersection points being transversal. 
\end{co}

\begin{proof} Since $\Hom(V,W)\setminus \mathcal{B}_X$ is connected, the pair $(f_0,f_1)$ can be connected  by a smooth path $\alpha:[0,1]\to \Hom(V,W)$  which intersects the wall only in regular points. Using a well-known  transversality principle (see \cite{DoK} p. 143) we find small perturbations of  $\alpha$ which coincide with $\alpha$ on a neighbourhood of $\{0,1\}$ and are transversal to the map $\mathcal{W}_X^0\hookrightarrow \Hom(V,W)$.
\end{proof}

Therefore Corollary \ref{generic-path} states that the difference 
$$\deg_{\nu(f_{1},\mu)}([f_{1}]_X)-\deg_{\nu(f_{0},\mu)}([f_{0}]_X)$$
 can always be  computed using such a path from $f_0$ to $f_1$ and the difference formula given by Corollary \ref{global}. Combining with Remark \ref{Phi}  one obtains the following   important general property of  the degree map $\deg_\mu^X:\pi_0(\Hom(V,W)\setminus{\cal W}_X)\to \Z$:
\begin{co} Let $\mu:X\times\det(W)\to \det(Y_X)$ be an orientation parameter. Then: 
\begin{enumerate}
\item All  values of the degree map $\deg_\mu^X:\pi_0(\Hom(V,W)\setminus{\cal W}_X)\to \Z$ are congruent modulo 2.
\item If $a\in \im(\deg_\mu^X)$,  then any integer $c$ with $-|a|\leq c\leq |a|$ which has the same parity as $a$ also belongs to $\im(\deg_\mu^X)$.
\end{enumerate}
\end{co}
\begin{proof} The first statement follows directly from the difference formula. For the second statement  use Corollary \ref{generic-path} and take into account   that 
\begin{itemize}
\item   the jump when crossing the wall transversally at a regular point is $\pm2$,
\item  the image of $\deg_\mu^X$ is invariant under the involution  $-\id_{\Z}$, by Remark \ref{Phi}.
\end{itemize}
\end{proof}
Note that  the  congruence  (1) does not follow  from Proposition \ref{estimates}  since $X$ is not supposed to be algebraic.
%\newpage
\subsection{Examples}
\subsubsection{Projecting a hyperquadric}
 Let $X$ be the  hyperquadric of $\P^n_\R$ ($n\geq 2$) defined by the equation $x_0^2-\sum_{i=1}^n x_i^2=0$. Dehomogenizing with respect to $x_0$ one gets an obvious identification  $X= S^{n-1}$. The relative orientability condition is always satisfied (it is obvious for $n\geq 3$ by Remark \ref{H1}).

Consider the two projections 
$$[f_0]:\P^n\setminus \{[1,0,\dots,0]\}\to\P^{n-1}\ ,\ [f_1]: \P^n\setminus \{[0,0,\dots,1]\}\to\P^{n-1}$$
given by 
$$[f_0]([x]):=[x_1,\dots,x_n]\ ,\ [f_1]([x]):=[x_0,\dots,x_{n-1}]\ .$$
Via the obvious identification $X= S^{n-1}$,  the first map is just the canonical projection $S^{n-1}\to\P^{n-1}$. The degrees of the restrictions $[f_0]_X:X\to\P^{n-1}$, $[f_1]_X:X\to\P^{n-1}$ with respect to a suitable choice of the trivialization $\mu$ are
$$\deg_{\nu(p_{0},\mu)}([f_{0}]_X)=2\ ,\ \deg_{\nu(p_{1},\mu)}([f_{1}]_X)=0\ .
$$
%.
\\ \\
\includegraphics[scale=0.8]{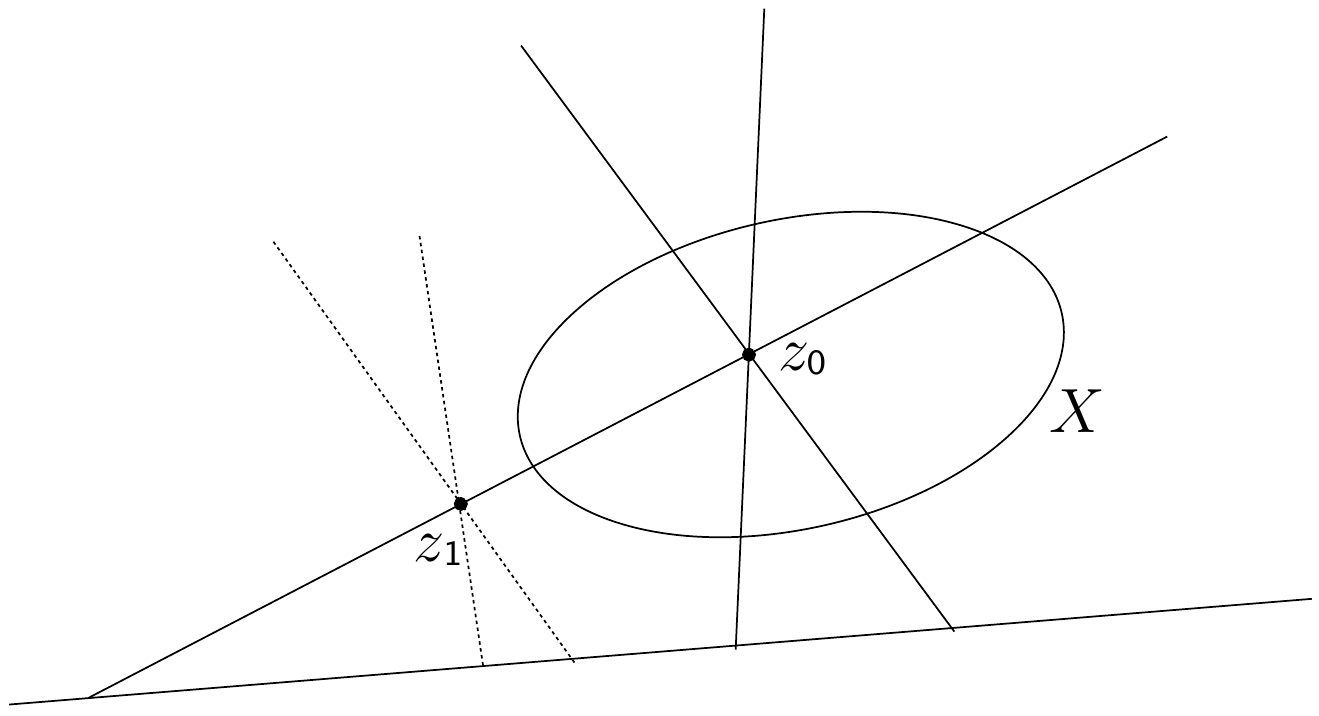}
\\ \\
The  picture  above shows an interesting phenomenon: the projection $[f_1]_X$, which has degree 0,  has some fibres  consisting of two points, which however come with opposite signs. For instance, the intersection of $X$ with the line connecting $z_0$ and $z_1$ consists of two points, which come with the same sign when considered as elements in the fibre of $[f_0]_X$, but with opposite signs when considered as elements in the fibre of $[f_1]_X$.
\\  
\subsubsection{Degree of   rational functions} \label{DegRatFunct}
 
 \def\Rat{\mathrm{Rat}}
 
 Let $_\R\F_n^*$ be the space of pairs of monic polynomials $(p,q)$ with real coefficients  of degree $n$ with no common factor. Writing
$$p(t)=t^n +\sum_{i=0}^{n-1} a_i t^i\ ,\  q(t)= t^n +\sum_{i=0}^{n-1} b_i t^i
$$
we see that $_\R\F_n^*$ can be identified with the Zariski open subset of $\R^{2n}$ defined by the condition   $\mathrm{Res}(p,q)\ne 0$.

A pair $(p,q)\in \ _\R\F_n^*$ defines a rational function
$$f_{pq}(t)=\frac{p(t)}{q(t)}\ ,
$$
hence $_\R\F_n^*$ can be identified with the space of rational functions of degree $n$ sending $\infty$ to 1. Such a rational function $f_{pq}$  has an extension $F_{pq}: \P^1_\R\to  \P^1_\R$. Taking the homogenizations $P\in \R[t_0,t_1]_{n}$, $Q\in \R[t_0,t_1]_{n}$ we see that $F_{pq}$ is given in homogeneous coordinates by   
$$F_{pq}(t_0,t_1):=[P(t_0,t_1),Q(t_0,t_1)]\ .$$

The set of possible degrees $\deg(F_{pq})$ when $(p,q)$ varies in the space $_\R\F_n^*$ is known. The result is given by Brockett's Theorem  (see Segal \cite{Se} p. 41 for an independent proof, or  \cite{By} Theorem 2.1).
\begin{thry} The space $_\R\F_n^*$ has $n+1$ connected components $_\R\F_{uv}^*$, where 
$$u+v=n\ ,\ u\geq 0\ ,\ v\geq 0$$
 and a pair $(p,q)\in\ _\R\F_n^*$ belongs to $ _\R\F_{uv}^*$ if  and only if  $\deg(F_{pq})=u-v$.
\end{thry}

Therefore  the set of possible degrees is $-n, -n+2,\dots,n-2,n$. Note that the degree of the map $F_{pq}^\C:\P^1_\C\to  \P^1_\C$ defined by a pair $(p,q)\in\ _\R\F_n^*$ is always $n$. We see that in this case {\it all}  values of the real degree  allowed by the general estimates and congruences in Proposition \ref{estimates} can occur.

The pair of polynomials corresponding to the rational function 
$$g_{uv}(t)=1+\sum_{i=1}^v\frac{-1}{t+i}+\sum_{j=1}^u \frac{1}{t-j}\ .
$$
is an element of $ _\R\F_{uv}^*$.

Note that the map $F_{pq}$  associated with a pair $(p,q)\in\ _\R\F_n^*$ can be  identified  with the composition $[\pi_{pq}]\circ v_n$, where $v_n:\P^1_\R\to \P^{n}_\R$ is the Veronese map  of degree $n$, and $\pi_{pq}:\R^{n+1}\to\R^2$ is the projection defined by $(P,Q)$. Therefore the theorem of Brockett identifies $n+1$ chambers in the complement 
$$\Hom(\R^{n+1},\R^2)_{{v_n(\P^1)}}= \Hom(\R^{n+1},\R^2)\setminus{\cal W}_{v_n(\P^1)}\ .$$

\section{Examples}

\subsection{Conjugation manifolds}

Let $(X,\tau)$ a topological space endowed with an involution such that $H^{\rm odd}(X,\Z_2)=0$. We recall from \cite{HHP}  that a cohomology frame for $(X,\tau)$ is a pair $(\kappa,\sigma)$, where
\begin{enumerate}
\item   $\kappa:H^{2*}(X,\Z_2)\to H^*(X^\tau,\Z_2)$ is a group isomorphism.
\item $\sigma:H^{2*}(X,\Z_2) \to H^{2*}_{\Z_2}(X,\Z_2)$ is a group morphism which is a section of the restriction map $\rho:H^{2*}_{\Z_2}(X,\Z_2)\to H^{2*}(X,\Z_2)$,
\end{enumerate}
such that the following conjugation equation holds:
$$r\circ \sigma(a)=\kappa(a) u^m + q(a)\ \forall a\in H^{2m}(X,\Z_2)\ .
$$
Here $r$ denotes the restriction map $H^*_{\Z_2}(X,\Z_2)\to H^*_{\Z_2}(X^\tau,\Z_2)=H^*(X,\Z_2)[u]$, and $q(a)$ is an element in $H^*(X,\Z_2)[u]$ whose degree with respect to $u$ is smaller than $m$.

If a cohomology frame for $(X,\tau)$ exists, then it is unique, natural with respect to equivariant maps, and  $\kappa$, $\sigma$  are automatically ring isomorphisms (not only group isomorphisms).

A $\Z_2$-space $(X,\tau)$ which admits a cohomology frame is called {\it conjugation space}. Examples of conjugation spaces are: all complex Grassmann  manifolds (with the  standard Real structure), all toric manifolds (with  their standard Real structure).

Let now $(X,\tau)$ be a paracompact conjugation space with cohomology frame $(\kappa,\sigma)$, and  let $( E,\tilde\tau)$ a Real complex vector bundle on $X$. Denote by $\bar c(E)$ the image of the total Chern class $c(E)$ in $H^*(X,\Z_2)$. Then $\kappa(\bar c(E))=w(E^{\tilde\tau})$ (see \cite{HHP} p. 950).

\begin{pr} Let  $(X,\tau)$, $(Y,\iota)$ be Real complex manifolds which are conjugation spaces with respect to their Real structures, and  let $f:X\to Y$ be a Real holomorphic map. Then $f$ is relatively orientable if and only if 
$f^*(c_1(Y))\equiv c_1(X)$ mod 2.
\end{pr}
\begin{proof} One has 
$$f(\R)^*w_1(T_{Y(\R)})=f(\R)^*(\kappa_Y (\bar c_1(T_Y)))$$
$$=\kappa_X(f^*(\bar c_1(T_Y)))=\kappa_X(c_1(T_X))=w_1(T_{X(\R)}),
$$
where $(\kappa_X,\sigma_X)$, $(\kappa_Y,\sigma_Y)$ are the cohomology frames of $(X,\tau)$ and $(Y,\iota)$ respectively, $T_X$, $T_Y$ the two tangent bundles regarded as complex vector bundles, $X(\R)$, $Y(\R)$ the fixed point loci, and $f(\R)$ the map $X(\R)\to Y(\R)$ induced by $f$. It suffices to apply Remark \ref{w}.
\end{proof}

\subsection{Plücker embeddings of real Grassmann manifolds}\label{plemb}

Let $V_0$ be a real vector space of dimension  $p+q$, $G_q(V_0)$ the Grassmann manifold
of $q$-planes in $V_0$. Take   $V=\wedge^q V_0$ and let $X$ be the image of the Plücker embedding $Pl:G_q(V_0)\to \P(V)$. Denoting by $U$ the tautological rank $q$ bundle on 
$G_q(V_0)$ we have a natural identification $T_{G_q(V_0)}=\Hom(U,\underline{V}_0/U)$, which shows that
$$w_1(G_q(V_0))=(pq+1) w_1(U)\ .
$$
On the other hand it is well-known that $Pl^*({\lambda_V})=\det(U)$. In our case we have $n=m+1=pq+1$.  Since on any Grassmann manifold one has $w_1(U)\ne 0$, we have 
\begin{co} A map $[f]_X:Pl(G_q(V_0))\to\P(W)$  associated with a linear map $f\in\Hom(\wedge ^q(V_0),W)$ satisfying  $\ker(f)\cap Pl(G_q(V_0))=\emptyset$ is relatively orientable if and only if $pq+1\equiv p+q$ (mod 2), i.e., iff $p$ and $q$ are not both even.
\end{co}

Consider now the special case when $V_0=S^{p+q-1} W_0^\vee$, where $W_0=\R^2$, and take $W:=S^{pq} W_0^\vee$. This special case is important because we have  a {\it standard}  linear epimorphism  
$$\varphi_{\scriptscriptstyle\rm  Wronski}:\wedge^q(S^{p+q-1} W_0^\vee)\to S^{pq} W_0^\vee\ ,\ \varphi_{\scriptscriptstyle\rm  Wronski}(F_1\wedge\dots \wedge F_ q):=W(F_1,\dots,F_q)\ ,
$$
where 
$$W: [S^{p+q-1}(W_0^\vee)]^q\to S^{pq} W_0^\vee$$
 denotes the homogeneous Wronskian. Alternatively, $\varphi_{\scriptscriptstyle\rm  Wronski}$ can be
defined as the composition of the {\it standard isomorphism} 
$$\wedge^q(S^{p+q-1} W_0^\vee)\textmap{\simeq} S^q(S^p W_0^\vee)$$
  with the natural projection $S^q(S^p W_0^\vee)\to S^{pq} W_0^\vee$  (see \cite{AC}).

Up to a constant factor the map $\varphi_{\scriptscriptstyle\rm  Wronski}$ can  also be obtained using the inhomogeneous Wronskian 
$$w: (\R[s]_{\leq p+q-1})^q\to \R[s]_{\leq pq} \ ,$$
via the obvious identifications $\R[s]_{\leq k}\simeq S^k(W_0^\vee)$ (see \cite{AC} section 2.8). With this remark   the results of \cite{EG1} (where the inhomogeneous Wronskian is used) apply. First, it is well known that $\P(\ker (\varphi_{\scriptscriptstyle\rm  Wronski}))\cap  G_q(S^{p+q-1} W_0^\vee)=\emptyset$, so we have a well-defined projection
$$[\varphi_{\scriptscriptstyle\rm  Wronski}]: G_q(S^{p+q-1} W_0^\vee)\to \P(S^{pq} W_0^\vee)
$$ 
for which the degree is known (see \cite{EG1}).

When $p$ and $q$ are not both even, then $[\varphi_{\scriptscriptstyle\rm  Wronski}]$ is relatively orientable and:
$$|\deg|[\varphi_{\scriptscriptstyle\rm  Wronski}]=\left\{
\begin{array}{ccc}
0&\rm if& p,\ q\hbox{ are both odd} \\
I(p,q)&\rm if &p+q\hbox{ is odd}\ . 
\end{array}\right. 
$$
Here $I(\cdot,\cdot)$ is   symmetric, and  for $2\leq p\leq q$ the integer $I(p,q)$ is given by:
$$I(p,q)=\frac{1! 2!\cdots (p-1)!(q-1)!(q-2)!\cdots (q-(p-1))!\big(\frac{pq}{2}\big)! }{(q-p+2)!(q-p+4)!\cdots (q+p-2)!\big(\frac{q-p+1)}{2}\big)!\big(\frac{q-p+3)}{2}\big)!\cdots \big(\frac{q+p-1)}{2}\big)!}\ .
$$
\begin{re} Using the first formula in Lemma \ref{pullBack} and Lemma 15 in \cite{OT2} one obtains a canonical isomorphism
$$\det(T_{Pl(G_q(S^{p+q-1}W_0^\vee))}^\vee)\otimes [f]^*(\det(T_{\P(S^{pq} W_0^\vee)})=$$
$$
[\det(W_0)^\vee]^{\otimes \frac{q(q-1)(p^2-2p-q)}{2}}\otimes [\det(U)^\vee]^{\otimes (p-1)(q-1)}\ ,
$$
for any $f\in\Hom(\wedge^q( S^{p+q-1} W_0^\vee), S^{pq}W_0^\vee)$  with   $\P(\ker(f))\cap Pl(G_q(S^{p+q-1} W_0^\vee))=\emptyset$. This shows that $[f]_{Pl(G_q(S^{p+q-1} W_0^\vee))}$ is canonically relatively oriented if and only if either $p$, $q\in 2\N+1$, or $p\in2\N+1$ and $q\in 4\N$, or $p\in 2\N$ and $q\in 4\N+1$.
\end{re} 
Therefore, if the pair $(p,q)$ satisfies one of these three conditions  then, for any regular value $[P]\in \P(S^{pq}W_0^\vee)$  of  the Wronski map $[\varphi_{\scriptscriptstyle\rm  Wronski}]$, one can associate an intrinsic sign to any element in the fibre $[\varphi_{\scriptscriptstyle\rm  Wronski}]^{-1}([P])$ (without having to orient the plane $W_0$). It would be interesting to have a geometric interpretation of these intrinsic signs.
\def\Quot{\mathrm{Quot}}

\subsection{Universal pole placement map}  Let $W_0$ be a real plane,  $\underline{V}_0$ the trivial bundle $\P(W_0)\times V_0$ on the projective line $\P(W_0)$ with fibre a $p+q$ dimensional vector space $V_0$. For  $\nu\in\N$ denote by   $\Quot^{p,\nu}_{\P(W_0)}(\underline{V}_0)$ the quot space of equivalence classes of quotients 
$s:\underline{V}_0\to Q$  of $\underline{V}_0$ with $\rk(\ker(s))=p$ and $\det(\ker(s))\simeq\mathcal{O}_{\P(W_0)}(-\nu)$. Every such quotient $s$ defines an element $QPl(s)\in \P(\wedge^p V_0\otimes S^\nu W_0^\vee)$  by  the formula
$$QPl(s):=[\wedge ^p k_s]\ .$$
Here $k_s\in\Hom(\ker(s),  \underline{V}_0)$ denotes the embedding of $\ker(s)$ in $\underline{V}_0$,  and 
$$\wedge ^p k_s \in H^0(\Hom(\wedge^p\ker(s), \wedge ^p \underline{V}_0))=H^0(\wedge^p\underline{V}_0\otimes (\wedge^p\ker(s))^\vee)$$
is the $p$-th exterior  power of $k_s$. Since  $\det(\ker(s))^\vee\simeq\mathcal{O}_{\P(W_0)}(\nu)$, the section $\wedge ^p k_s$ of $\wedge^p\underline{V}_0\otimes (\wedge^p\ker(s))^\vee$ defines an element of $H^0(\wedge^p\underline{V}_0(\nu))=\wedge^p V_0\otimes S^\nu W_0^\vee$, which is  well defined up to multiplication with a non-vanishing scalar. Recall that we have a perfect paring $\wedge^q V_0\times \wedge^p V_0\to \det V_0$, which induces an isomorphism 
$$\wedge^p V_0\to (\wedge^q V_0)^\vee\otimes\det(V_0)=\Hom(\wedge^q V_0,\det(V_0))\ .$$

Therefore the element $QPl(s)=[\wedge ^p k_s]$ can be regarded as an element of 
$$\P(\Hom(\wedge^q V_0,S^\nu W_0^\vee\otimes\det(V_0)))=\P(\Hom(\wedge^q V_0,S^\nu W_0^\vee))$$
 as claimed. Note that in general the map 
$$QPl: \Quot^{p,\nu}_{\P(W_0)}(\underline{V}_0)\to \P(\Hom(\wedge^q V_0,S^\nu W_0^\vee))$$
 is not an embedding.

A quotient $[s]\in \Quot^{p,\nu}_{\P(W_0)}(\underline{V}_0)$ defines a central projection 
$$\psi_{[s]}:\P(\wedge^q V_0)\setminus \P(\ker(\wedge^p k_s)) \to  \P(S^\nu W_0^\vee)\ .$$
Since $\P(\wedge^q V_0)$ contains the  image of the Plücker embedding  %
 $$Pl:G_q(V_0)\to \P(\wedge^q V_0)\  ,$$
  it is interesting to study the composition
$$\phi_{[s]}:=\psi_{[s]}\circ Pl:G_q(V_0)\setminus \P(\ker(\wedge^p k_s))\to \P(S^\nu W_0^\vee)
$$
associated with an element $[s]\in \Quot^{p,\nu}_{\P(W_0)}(\underline{V}_0)$. When $\nu=pq$ we can ask if   this projection is defined on the all of  $G_q(V_0)$, and  when $p$ and $q$ are not both even,  so that $\phi_{[s]}$ is relatively orientable,  one can also ask for the degree.
\begin{re} Consider  the special case where $V_0=S^{p+q-1} W_0^\vee$ and $W_0=\R^2$. It seems to be well known \cite{EG3} that in this case there exists an element $$s_{\scriptscriptstyle\rm  Wronski} \in \Quot^{p,pq}_{\P(W_0)}(\underline{V}_0)$$ such that  
$$\varphi_{[s_{\scriptscriptstyle\rm  Wronski}]}=[\varphi_{\scriptscriptstyle\rm  Wronski}]$$
is the Wronski projection introduced in section \ref{plemb}.

\end{re}

Note that the chamber structure of  $\Hom(\wedge^q V_0,S^{pq} W_0^\vee)_{G_q(V_0)}$ induces  a chamber structure on the complement 
$$\Quot^{p,pq}_{\P(W_0)}(\underline{V}_0)\setminus QPl^{-1}(\P({\cal W}_{G_q(V_0)}))$$
of the  pull-back of the projectivized wall $\P({\cal W}_{G_q(V_0)}))$  via $QPl$.

The image $\phi_{[s]}(U)\in \P(S^{pq} W_0^\vee)$ of a $q$-plane $U\in G_q(V_0)\setminus \P(\ker(\wedge^p k_s))$   can be explicitly described as follows: Denote by $j_U:\underline{U}\to\underline{V}_0$ the obvious embedding of the  trivial rank $q$-bundle $\underline{U}:=\P(W_0)\times U$ in  $\underline{V}_0$, and by $\rho_U:\underline{V}_0\to \underline{V}_0/\underline{U}$ the projection onto the quotient bundle. The  determinant $\det(\rho_U\circ k_s)$ of the composition $\rho_U\circ k_s:\ker(s)\to \underline{V}_0/\underline{U}$ can be regarded as an element of $\wedge^p(V_0/U)\otimes S^{pq} W_0^\vee $.  If this element is non-zero it defines an element 
$$PP[s](U)\in \P(\wedge^p(V_0/U)\otimes S^{pq} W_0^\vee)=\P(S^{pq} W_0^\vee )\ ,$$ 
called the {\it pole placement} of $[s]\in \Quot^{p,pq}_{\P(W_0)}(\underline{V}_0) $ at $U\in G_q(V_0)$.  On the other hand, one can easily prove  that $\det(\rho_U\circ k_s)$ is non-zero if  and only if  $\phi_{[s]}$ is defined at $U$.  
\begin{lm}\label{comm} When $\phi_{[s]}$ is defined at $ U$, one has 
\begin{equation}\phi_{[s]}(U)=PP[s](U)\ .
\end{equation}
\end{lm}
\begin{proof} Consider the rank $p$ vector bundle $K_s$ on $\P(W_0)$ associated with the sheaf $\ker(s)$,  choose $x\in\P(W_0)$ and vectors $l_1,\dots,l_p\in K_{s,x}$. Let $(v_1,\dots,v_p,v_{p+1},\dots v_{p+q})$ be a basis of $V_0$ such that $(v_{p+1},\dots,v_{p+q})$ is a basis of $U$.  Let $U^\bot:=\langle v_1,\dots,v_p\rangle$ and $\bar k_s$ the composition of $k_s$ with the projection on $U^\bot$. Regarding $\wedge^p k_s$ as an element in 
$$H^0(\P(W_0),{\cal H}om(\det(K_s),
{\cal H}om(\wedge^q \underline{V_0}, \det(\underline{V}_0)))=\Hom(\wedge^q V_0,\wedge^{p+q} V_0)\otimes S^\nu W_0^\vee $$
we have
$$(\wedge^p k_s)(l_1\wedge\dots\wedge l_p)(v_{p+1}\wedge\dots\wedge v_{p+q})=k_s(l_1)\wedge\dots\wedge k_s(l_p)\wedge v_{p+1}\wedge\dots\wedge v_{p+q}=
$$
$$=\bar k_s(l_1)\wedge\dots\wedge \bar k _s(l_p)\wedge v_{p+1}\wedge\dots\wedge v_{p+q}=[\det(\rho_U\circ k_s)(l_1\wedge\dots\wedge l_p)]\otimes (v_{p+1}\wedge\dots\wedge v_{p+q})\ ,
$$
where in the last  equality we have used the canonical isomorphism $\det(V_0)=\det(V_0/U)\otimes\det(U)$. Therefore one obtains the  following equality in the space  $S^\nu W_0^\vee\otimes\det(V_0)$ 
$$\wedge ^p k_s(v_{p+1}\wedge\dots\wedge v_{p+q})=\det(\rho_U\circ k_s)\otimes (v_{p+1}\wedge\dots\wedge v_{p+q})\ ,
$$
which shows that $\wedge ^p k_s(v_{p+1}\wedge\dots\wedge v_{p+q})$ and $\det(\rho_U\circ k_s)$ define the same element in $\P(S^\nu W_0^\vee )$.
\end{proof}

For a quotient $[s]\in  \Quot^{p,pq}_{\P(W_0)}(\underline{V}_0)$ the assigment $U\mapsto PP[s](U)$ defines a rational map $G_q(V_0)\dasharrow \P(S^{pq} W_0^\vee)$. We denote by $\mathrm{Rat}(G_q(V_0),\P(S^{pq} W_0^\vee))$ the set of such rational maps.  Letting $[s]$ vary in $ \Quot^{p,pq}_{\P(W_0)}(\underline{V}_0)$ we obtain  a map 
$$PP:  \Quot^{p,pq}_{\P(W_0)}(\underline{V}_0)\to \mathrm{Rat}(G_q(V_0),\P(S^{pq} W_0^\vee))\ ,$$
which we call the universal pole placement map.    Lemma \ref{comm} shows that the following diagram is commutative:

$$
\begin{array}{c}
\unitlength=1mm
\begin{picture}(42,32)(-4,-24)
\put(-15,4){$\Quot^{p,pq}_{\P(W_0)}(\underline{V}_0)$}
\put(10,5){\vector(2,0){19}}
\put(32,4){$\P(\Hom(\wedge^q V_0,S^{pq}W_0^\vee))$}
\put(-3,1){\vector(0, -3){18}}
\put(-22, -22){$\mathrm{Rat}(G_q(V_0),\P(S^{pq} W_0^\vee))$}
\put(-2,-8){$PP$}
\put(33,-10){$Pl^*$}
\put(18,7){$QPl$}
\put(43,1){\vector(-3, -2){28}}

\end{picture} 
\end{array} 
$$
\subsection{ A real subspace problem} 

Let $W_0$, $V_0$ be   real vector spaces of dimensions 2 and   $p+q$ respectively. We denote by $S^{pq}_0(\P(W_0))$ the subset of elements $\sg$ in the symmetric power $S^{pq}(\P(W_0))$ consisting of pairwise distinct points.  
\\ \\
{\bf Subspace problem:} Fix a regular algebraic map $\gamma:\P(W_0)\to G_p(V_0)$ of algebraic degree $pq$, and an element $\sg\in S^{pq}_0(\P(W_0))$. Count  the $q$-dimensional linear subspaces $U\subset V_0$ such that 
$$\P(U)\cap \P(\gamma(\xi))\ne \emptyset\ \ \forall\xi\in \sg\ .\eqno{(S_{\gamma,\sg})} $$
 \vspace{2mm}

We will show that this problem has an interesting interpretation which, for general  $\sg\in S^{pq}(\P(W_0))$, allows one to associate a sign to every  $q$-dimensional subspace $U$ satisfying the condition $(S_{\gamma,\sg})$ and to compute the total number of solutions of $(S_{\gamma,\sg})$ when these signs are taken into account. Indeed, for any regular algebraic map $\gamma:\P(W_0)\to G_p(V_0)$ there exists a {\it bundle epimorphism}
$$s_\gamma:\underline{V}_0\to Q_\gamma
$$
classifying $\gamma$, i.e., such that
$$\gamma(\xi)=\ker(s_{\gamma,\xi})\ \ \forall \xi\in\P(W_0)\ .
$$
The equivalence class $[s_\gamma]\in \Quot^{p,pq}_{\P(W_0)}(\underline{V}_0)$ is well defined, and the assigment $\gamma\mapsto [s_\gamma]$ defines an embedding %
$$\mathrm{Mor}^{pq}(\P(W_0),G_p(V_0))\to \Quot^{p,pq}_{\P(W_0)}(\underline{V}_0)\ ,$$
 where $\mathrm{Mor}^{pq}(\P(W_0),G_p(V_0))$ is the space of regular algebraic morphisms $\P(W_0)\to G_p(V_0)$ of algebraic degree $pq$. Consider the pole placement map $\phi_{[s_\gamma]}:G_q(V_0)\to \P(S^{pq} W_0^\vee)$ corresponding to  the quotient $[s_\gamma]$.

Let $P_\sg\in \R[W_0]_{pq}=S^{pq} W_0^\vee$ be a  homogeneous  polynomial of degree $pq$ on $W_0$ whose set of roots coincides with $\sg$.  This polynomial is   determined by $\sg$ up to multiplication by a non-vanishing constant.  
\begin{lm} The set of solutions of the problem $(S_{\gamma,\sg})$ can be identified with the fibre $\phi_{[s_\gamma]}$ over $[P_\sg]$:
$$\phi_{[s_\gamma]}^{-1}([P_\sg])=\{U\in G_q(V_0)|\ \P(U)\cap \P(\gamma(\xi))\ne \emptyset\ \ \forall\xi\in \sg\}\ .
$$
\end{lm}
\begin{proof} Indeed, using the definition of $\phi_{[s]}$, it follows that for any $U\in G_q(V_0)$ the set of zeros  of a polynomial $P_U\in  \R[W_0]_{pq}$  representing $\phi_{[s_\gamma]}(U)\in\P(S^{pq} W_0^\vee)$ coincides with the set of  points $\xi\in \P(W_0)$ for which $\det(\rho_U\circ k_{s_\gamma,\xi})\ne 0$, i.e., with the set of points $\xi\in \P(W_0)$ for which $\gamma(x)\cap U\ne\{0\}$. If the latter set  is $\sg$ (which has maximal cardinal $pq$) this condition is equivalent to $[P_U]=[P_\sg]$. Therefore $\phi_{[s_\gamma]}(U)=[P_{\sg}]$ if and only if $U$ satisfies the condition $(S_{\gamma,\sg})$.
\end{proof}

\begin{co} Suppose that $p$ and $q$ are not both even and let $\gamma:\P(W_0)\to G_p(V_0)$ be a regular algebraic map of algebraic degree $pq$ such that  
$$[s_\gamma]\in \Quot^{p,pq}_{\P(W_0)}(\underline{V}_0)\setminus QPl^{-1}(\P({\cal W}_{G_q(V_0)}))\ .$$
Then the pole placement map $\phi_{[s_\gamma]}:G_q(V_0)\to\P(S^{pq}W_0^\vee)$ is well defined.  Choose a relative orientation $\nu$ of  $\phi_{[s_\gamma]}$. There exists an open  dense subset   $\Sg_\gamma\subset S^{pq}_0(\P(W_0))$ such that for any $\sg\in \Sg_\gamma$ one has
$$\sum_{U\hbox{ solves }(S_{\gamma,\sg})} \varepsilon_{U,\nu}=\deg_\nu(\phi_{[s_\gamma]})\ .
$$
\end{co}
\begin{proof} The map $\sg\mapsto [P_\sg]$ identifies $S^{pq}_0(\P(W_0))$ with an open subset of $\R[W_0]_{pq}=S^{pq} W_0^\vee$. It suffices to apply Sard theorem to the map $\phi_{[s_\gamma]}$ and to take into account that the set of regular values of a proper smooth map is always open.
\end{proof}

%\subsection{Canonical relative orientation of projections associated with functors  }

%Correcting $W=G(W_0)$ with a tensor power of $\det(W_0)$ in order to produce a canonical %orientation.
%\subsection{Projective toric varieties}

{\ }
\vspace{10mm}  \\
{\small Christian Okonek: \\
Institut f\"ur Mathematik, Universit\"at Z\"urich,
Winterthurerstrasse 190, CH-8057 Z\"urich,\\
e-mail: okonek@math.uzh.ch
\\  \\
Andrei Teleman: \\
CMI,   Aix-Marseille Universit\'e,  LATP, 39  Rue F. Joliot-Curie,  13453
Marseille Cedex 13,   e-mail: teleman@cmi.univ-mrs.fr
}

\end{document}